\documentclass[a4paper, 11pt, dvipsnames]{extarticle}

\usepackage{bold-extra}
\usepackage[utf8]{inputenc}    
\usepackage[T1]{fontenc}
\usepackage[english]{babel}
\usepackage[normalem]{ulem}
\usepackage{dsfont, bm,bbm}
\usepackage{ifthen}
\usepackage{float, graphicx, caption, wrapfig, tikz}
\usetikzlibrary{backgrounds, calc}
\usepackage{amsthm, amsmath, amsfonts, amssymb, mathrsfs, stackrel, mathtools}
\usepackage[alphabetic]{amsrefs}

\usepackage{hyperref, xcolor, titlesec} 
\hypersetup{colorlinks = true, urlcolor = blue,linkcolor=BrickRed, citecolor=NavyBlue, bookmarksopen = true}
\usepackage[nomarginpar]{geometry}
\numberwithin{equation}{section}

\newcommand{\e}{{\varepsilon}}
\newcommand{\E}{\mathds{E}}

\DeclareMathOperator{\Tr}{Tr}

\let \d \relax
\newcommand{\d}{\mathrm{d}}

\let\O\relax
\newcommand{\O}[1]{\mathcal{O}\left(#1\right)}

\newcommand{\unn}[2]{[\![#1,#2]\!]}

\DeclareMathOperator{\kurt}{Kurt}

\def\bet{\begin{theorem}}
\def\eet{\end{theorem}}
\def\bel{\begin{lemma}}
\def\eel{\end{lemma}}
\def\bas{\begin{ass}}
\def\eas{\end{ass}}
\def\bec{\begin{cor}}
\def\eec{\end{cor}}
\def\bed{\begin{defn}}
\def\eed{\end{defn}}
\def\bep{\begin{prop}}
\def\eep{\end{prop}}
\def\beq{\begin{equation}}
\def\eeq{\end{equation}}
\def\bea{\begin{equation*}}
\def\eea{\end{equation*}}

\def\bex{\begin{ex}}
\def\eex{\end{ex}}

\def\1{{\mathbbm 1}}

\def\benr{\begin{enumerate}[label=(\roman*)]}
\def\eenr{\end{enumerate}}
\def\A{\mathcal{A}}

\def\R{\mathbb{R}}

\newcommand{\bma}{\begin{bmatrix}}
\newcommand{\ema}{\end{bmatrix}}

\theoremstyle{plain}
\newtheorem{theorem}{Theorem}[section]
\newtheorem{lemma}[theorem]{Lemma}

\newtheorem{proposition}[theorem]{Proposition}
\theoremstyle{definition}

\newtheorem{definition}[theorem]{Definition}

\titleformat{\paragraph}[runin]{\itshape\normalsize}{\theparagraph}{}{}
\titleformat{\subparagraph}[runin]{\itshape\normalsize}{\theparagraph}{0em}{}
\titleformat{\section}[block]{\normalfont\filcenter}{\Large\thesection .}{.7em}{\Large\scshape}
\titleformat{\subsection}[runin]{\normalfont}{\large \bf \thesubsection .}{.5em}{\large\bf}[.]
\titleformat{\subsubsection}[runin]{\normalfont}{\bf \thesubsubsection .}{.5em}{\bf}[.]

\begin{document}

\title{\scshape\bfseries{Largest Eigenvalues of the Conjugate Kernel of Single-Layered Neural Networks.}}
\author{L. \textsc{Benigni}\thanks{This material is based upon work supported by the National Science Foundation under Grant No. DMS-1928930 while L. Benigni participated in a program hosted by the MSRI in Berkeley, California during the Fall 2021 semester.}\\\vspace{-0.15cm}\footnotesize{\it{University of Chicago}}\\\footnotesize{\it{lbenigni@uchicago.edu}}\and S. \textsc{P\'ech\'e}\thanks{Research was accomplished while supported by the Institut Universitaire de France.}\\\vspace{-0.15cm}\footnotesize{\it{LPSM, Universit\'e de Paris}}\\\footnotesize{\it{peche@lpsm.paris}}}
\date{}
\maketitle

\tikzset{every loop/.style={min distance=10mm,in=60,out=120,looseness=10}}
\tikzset{every node/.style={draw,circle, scale=1}}
\tikzset{every label/.append style={font=\selectfont}}
%\fontsize{9}{10.8}
\begin{abstract}
\small{This paper is concerned with the asymptotic distribution of the largest eigenvalues for some nonlinear random matrix ensemble stemming from the study of neural networks.
More precisely we consider $M= \frac{1}{m} YY^\top$ with $Y=f(WX)$ where $W$ and $X$ are random rectangular matrices with i.i.d. centered entries. This models the data covariance matrix or the Conjugate Kernel of a single layered random Feed-Forward Neural Network.
The function $f$ is applied entrywise and can be seen as the activation function of the neural network.
We show that the largest eigenvalue has the same limit (in probability) as that of some well-known linear random matrix ensembles. In particular, we relate the asymptotic limit of the largest eigenvalue for the nonlinear model to that of an information plus noise random matrix, establishing a possible phase transition depending on the function $f$ and the distribution of $W$ and $X$. This may be of interest for applications to machine learning. }
\end{abstract}

\section{Introduction}

This article is concerned with the asymptotic behavior of the largest eigenvalue of the data covariance matrix of random feed-forward neural networks arising in machine learning.
Artificial neural networks have been developped in the late fifties  to give a mathematical modelisation of the brain behavior. They are nowadays used via machine learning in many fields of applications such as langage recognition, computer vision. We mention among other applications in image or speech recognition \cites{krizhevsky2012imagenet, hinton2012deep} or translation \cite{wu2016google}). It is also used now in video, style transfer, dialogues, games and countless other topics. We refer to \cite{Zdeborova} and \cite{schmidhuber2015deep} for an overview of the subject. 
Yet understanding the mathematical framework behind  learning is still missing. The main difficulty comes from the complexity of studying  non-convex functions of a very large number of parameters \cites{choromanska2015loss, pennington2017geometry}.  

An artificial neural network can be modeled as follows: some input column vector $x\in \mathbb{R}^{n_0}$ goes through a multistage architecture of alternated layers with both
linear and non linear functionals: let $g_i: \mathbb{R} \to \mathbb{R}$, $i=1, \ldots, L$ be some given \emph{activation  functions} and $W_i,i=1\ldots L$ be $n_i\times n_{i-1}$ matrices. The output vector after layer $L$ is 
\begin{equation}s_1=g_1(W_1x), \quad s_i=g_i(W_is_{i-1}), i=2, \ldots,L .\label{FFNN}\end{equation}
The functions $g_i$ are here applied componentwise. 
The matrices $W_i$ are the (synaptic) weights in the  layer $i$ and the activation function $g_i$ models the impact of the neurons in the architecture. 
Commonly used activation functions are $g(x)=\max(0,x)$ (known as the ReLU activation function for Rectified Linear Unit) or the sigmoid function $g(x)=(1+\e^{-x})^{-1}$. The parameter $L$ is called the depth of the neural network. 
In this article we are interested in the so-called extreme learning, which corresponds to a one layer network, and more precisely in the initialisation phase. We also make the assumption that the weight matrices $W_i$ are random: this is a usual simplifying assumption to try to understand the model in a rigorous framework but also models the weights at initialization.

Generally in supervised machine learning, one is given a $n_0\times m$ matrix dataset $X$ coinjointly with a target dataset $Z$ of size $d\times m$ to train the network. The parameter $m$ is here the sample size. The aim is to determine a function $h$ so that, given for instance a new data matrix $X'$, the output of the function $h(X')$ yields an acceptable approximation of the target. The parameters to be learned are here the weight matrices.  The error of the approximation in the training phase is measured through a loss function. 
In the context of Feed Forward Neural Networks as in (\ref{FFNN}), one of the commonly used learning method in high dimension is ridge regression (and $h$ is linear in $g_1(W_1X)$).   More precisely, in the one layer case ($L=1$) the loss function is 
$$B\in \mathbb{R}^{d\times n_1}\mapsto \mathcal{L}(B):= \frac{1}{2dm}||Z - B^\top (g_1(W_1 X))||_F^2 +\gamma || B||_F^2,$$ where $\gamma$ is a penalizing parameter.  The optimal matrix $B$ can then be proved to be proportional to $Y QZ^\top$ where $Y=(g_1(W_1 X))$ and 
\begin{equation}
	\label{defQ}Q=\left( \frac{1}{m}Y^\top Y +\gamma I\right)^{-1}. \end{equation}
 As a consequence, the performance of this learning procedure can be measured thanks to the asymptotic spectral properties of the matrix $M=\frac{1}{m}Y^\top Y $ which is called the Conjugate Kernel of the network. Indeed, for the one layer case, the expected training loss
can be proved to be related to the asymptotic e.e.d. (and Stieltjes transform) of $M$. It is given by 
\[
\mathbb{E} (\mathcal{L}(B))= -\frac{\gamma^2}{m}\frac{\partial}{\partial \gamma} \mathbb{E} (\text{Tr } Q),
\]
where $Q$ is given by (\ref{defQ}) and $\text{Tr}$ denotes the unnormalized trace. Here the expected value is evaluated with respect to the distribution of the weight matrix $W.$ Besides, the largest eigenvalues are also of interest since the training occurs most rapidly along the eigenvectors of the largest eigenvalues \cite{advani2020high}.

After the training phase, the testing phase brings the new data matrix $X'$ and uses the previous optimal $B$ to estimate the target. It is then compared to the target (again unused in the training phase). The testing performance has been shown numerically to be improved in the presence of outlying eigenvalues in \cite{LouartLiaoCouillet}. This may come from the fact that outlying eigenvalues and corresponding eigenvectors may bear some features characterizing the general type of the data $X$ or  $X'$. In particular, they showed numerically that the position of outliers depends on the distribution of $W$ and that the further away an outlying eigenvalue is the better the testing error. The behavior of extreme eigenvalues is also of interest for comparing different acceleration methods in the stochastic gradient methods used in the learning procedure (see e.g. \cite{Paquette}). More generally, theoretical results about kernel methods are of interest in machine learning in order to select the appropriate kernel as explained e.g. \cite{Jacot_etal}. This indeed depends on the task which is expected (see e.g. \cite{20}, \cite{10}, \cite{16}). In particular universality results are of special interest so as to get rid of the particular features of the distribution. Regarding applications to machine learning questions, the existence of outliers seems an artifact of the non linearity of the model. It gives some structural information on the neural network even if it seems difficult to extract some precise information on this.

The study of kernel matrices similar to our model  goes back to \cite{ElKaroui} and \cite{Singer}, regarding the asymptotic behavior of the empirical eigenvalue distribution. In these articles where two different matrix models are considered, the limiting e.e.d. is shown in both cases to coincide with that of a linear random matrix ensemble. Regarding the edge of the spectrum, the behavior of extreme eigenvalues for kernel matrix models has been investigated  in \cite{MontanariFan}. Therein the authors show that for odd activation functions, the largest eigenvalue of some kernel matrix sticks to the edge of the limiting e.e.d. which is proved to be the free convolution of the Wigner semicircle law and the Marcenko--Pastur one. Therein it is also shown that a symmetrized oulier may exit the support of the limiting e.e.d. for an even polynomial activation function. Recently \cite{Couillet21} have studied an analog of the model we consider: they derive the limiting behavior of the largest eigenvalues too for a similar model to the one studied in this paper except that they consider the expected kernel rather than the kernel itself when the data $X$ is specifically given by a mixture of Gaussians. They also give the asymptotic eigenspace of the largest eigenvalue which we cannot do in this paper.

\paragraph{}In the recent article \cite{BePe}, following \cite{PeWo}, we have investigated the asymptotic empirical eigenvalue distribution (e.e.d) for a nonlinear matrix ensembles of the form 
\[ 
Y=f\left (\frac{WX}{\sqrt n_0}\right )
\]
where both $W$ and $X$ are random matrices with i.i.d. centered entries and respective dimension $n_1\times n_0$ and $n_0\times m$. This is a toy model to understand extreme learning as both weights and data are taken random. The behavior of the spectrum is then quite well understood now in the large $n_0$ limit provided $n_1, n_0, m$ grow to infinity with the same speed.
The behavior of extreme eigenvalues had then been postponed as numerical simulations show that, quite intriguingly, due to the nonlinearity, some large eigenvalues may separate from the bulk of the spectrum.  In view of simulations, the presence of a certain number of outliers seems to depend on the distribution of the entries of $W$ and $X$, the dimensions of the matrix $n_0, n_1, m$  and the function $f$. This is the question we here investigate, establishing a threshold where some eigenvalues separate, giving some ideas of why one or two such eigenvalues may separate and explicit the role of the dimension. 
 We also mention that, the step further is to extend our result to the case of deterministic data $X$ as for the limiting e.e.d. This has been obtained by  \cite{Zhou} in the latter case and we devote the study of the edge of the spectrum to future work.

The article is organized as follows. Section \ref{sec:2} states our main results, about a possible separation of two eigenvalues. It also gives the main ideas for the proof. 
Section \ref{sec:3} is concerned with the behavior of the largest eigenvalues for an odd activation function where the largest eigenvalues stick to the bulk. Section \ref{sec:4} deals with the case of an even activation function where the limiting e.e.d. is the Marchenko--Pastur distribution. Last Section \ref{sec:5} is concerned with the general case by combining the two previous sections, where we cannot describe the asymptotic behavior of the largest eigenvalues in full generality but compares it to a linear model.
\section{Model and results \label{sec:2}}
Suppose $W$ is a $n_1\times n_0$ random matrix with i.i.d entries and $X$ is a $n_0\times m$ random matrix with i.i.d entries. The distributions of the entries of both matrices $W$ and $X$ are assumed to be centered: $\E X_{11}=\E W_{11}=0$ and with vanishing third moment $\E X_{11}^{3}=\E W_{11}^3=0$. \\
We also need the following assumption on the tails of $W$ and $X$: there exist constants $\vartheta_w,\,\vartheta_x>0$ and $\alpha>1$ such that for any $t>0$ we have

\begin{equation}\label{eq:assumwx}
\mathds{P}\left(
	\left\vert W_{11}\right\vert > t
\right)
\leqslant
e^{-\vartheta_w t^\alpha}
\quad\text{and}\quad
\mathds{P}\left(
	\left\vert X_{11}\right\vert > t
\right)
\leqslant
e^{-\vartheta_x t^\alpha}.
\end{equation}

We introduce the following notations, including the kurtoses of the distributions:
\[
\begin{gathered}
	\sigma_w^2\coloneqq\E{W_{11}^2},
	\quad
	\mu_{4,w}\coloneqq\E{W_{11}^4},
	\quad
	\kurt(W)=\frac{\mu_{4,w}}{\sigma_{w}^4}. 
	\\
	\sigma_x^2\coloneqq\E{X_{11}^2},
	\quad
	\mu_{4,x}\coloneqq\E{X_{11}^4},
	\quad
	\kurt(X)=\frac{\mu_{4,x}}{\sigma_{x}^4}.
\end{gathered}
\]
We also suppose that we are in the following random matrix regime: as $n_0\to \infty$ one has that
\[
\frac{n_0}{m}\to \phi,\quad \frac{n_0}{n_1}\to \psi,\quad\text{and}\quad\frac{n_1}{m}\to \gamma \coloneqq \frac{\phi}{\psi}.
\]
As an additional assumption, we also suppose that there exist positive constants $C_f$ and $c_f$ and $A_0>0$ such that for any $A\geqslant A_0$ and any $n\in\mathbb{N}$ we have,
\begin{equation}\label{eq:assum2f}
\sup_{x\in[-A,A]}
\vert f^{(n)}(x)\vert
\leqslant
C_fA^{c_fn}.
\end{equation}
These are very strong assumptions on the function $f$, in particular $f$ is real analytic, but this should be compared with the weaker ones we make on the distributions of $W$ and $X$ as we do not use any Gaussian concentration bounds as in \cites{LouartLiaoCouillet, Zhou} for instance.
Finally we define the following three parameters for a function $f$ such that $\E{[f(\sigma_w\sigma_x\mathcal{N}(0,1))]}=0$:
\begin{equation} \label{defipara}
\begin{gathered}
\theta_1(f)
=
\E{[f(\sigma_w\sigma_x\mathcal{N}(0,1))^2]},
\quad
\theta_2(f)
=
\E{\left[\sigma_w\sigma_xf'(\sigma_w\sigma_x\mathcal{N}(0,1))\right]}^2,
\\
\theta_3(f)
=
\E{\left[\frac{(\sigma_w\sigma_x)^2}{2}f''(\sigma_w\sigma_x \mathcal{N}(0,1))\right]}^2.
\end{gathered}
\end{equation}
\begin{remark} Note that a non centered distribution for $W$ and $X$ or removing the ``centering'' of the function $f$ given above may result in the existence of very large eigenvalues which are typicaly non-informative. These assumptions are crucial for our results.
\end{remark}
We give some relationships between these parameters in the following lemma.

\bel\label{l:param}
We have for any functions such that the parameters exist:
\begin{enumerate}
	\item $\theta_2(f)\leqslant \theta_1(f)$ with equality if, and only if, $f(x)=\alpha x$ for some $\alpha\in\mathbb{R}$.
	\item $\theta_3(f)\leqslant\frac{1}{2}\theta_1(f)$ with equality if, and only if, $f(x)=\alpha(x^2-\sigma_w^2\sigma_x^2)$ for some $\alpha\in\R$.
\end{enumerate}
\eel
\begin{proof}
This is an application of Stein's lemma.
\end{proof}
It has been shown in \cite{BePe} that the e.e.d. of the non linear random matrix
\[
M= \frac{1}{m}YY^\top\in\R^{n_1\times n_1} \quad\text{with}\quad Y_{ij}= f\left (\frac{(WX)_{ij}}{\sqrt n_0}\right )\quad\text{for }i\in\unn{1}{n_1}\text{ and }j\in\unn{1}{m}\]
 converges under the above assumptions to a deterministic probability measure $\mu$.
\bet[\cite{BePe}]\label{theo:bepe} Denote $\lambda_{n_1}\leqslant \dots\leqslant \lambda_1$ the eigenvalues of $M$, there exists a deterministic probability measure $\mu$ depending only on $\theta_1(f),\,\theta_2(f),\,\phi,$ and $\psi$ such that
\[
\frac{1}{n_1}\sum_{i=1}^{n_1}\delta_{\lambda_i}\xrightarrow[n_1\to\infty]{} \mu.
\]
\eet
 The distribution $\mu$ has been characterized in a series of papers (\cite{BePe}, \cite{Pe}, \cite{Zhou}). 
In particular it admits a compact support whose edges are denoted by $\bf{u_-}$ for the bottom edge and $\bf{ u_+}$ for the top edge. It is shown in \cite{BePe} that given an integer $q$, one has that
\begin{equation}
\int x^q d\mu= \mathfrak{m}_q(f),
\end{equation}
where 
\begin{equation}\label{eq:defmoment}
\mathfrak{m}_q(f)
:=
\sum_{I_i,I_j=0}^q\sum_{b=0}^{I_i+I_j+1}{\mathcal{A}}(q,I_i,I_j,b)\theta_1(f)^b\theta_2(f)^{q-b}\psi^{I_i+1-q}\phi^{I_j}.
\end{equation}
In the above formula, ${\mathcal{A}}(q,I_i,I_j,b)$ denotes the number of so-called \emph{admissible} graphs of length $2q$ with $b$ cycles of length $2$ with a given number of identifications. An admissible graph is here a bipartite connected graph with $2q$ edges, all of which belonging to exactly one cycle.
When $\theta_2(f)=0$, it has also been shown that the limiting distribution $\mu$ is the Marchenko--Pastur distribution with shape $\gamma=\frac{\phi}{\psi}$ and variance $\theta_1(f)$. More precisely, it has been proved in \cite{Pe} that $\mu$ is also given by the asymptotic empirical eigenvalue distribution of the linear model:
\begin{equation}\label{eq:linear}
\mu=\mathrm{limspec}\left[\frac{1}{m}
\left(
	\sqrt{\theta_1-\theta_2}\tilde{Z}+\sqrt{\frac{\theta_2}{n_0}}\tilde{W}\tilde{X}
\right)
\left(
\sqrt{\theta_1-\theta_2}\tilde{Z}+\sqrt{\frac{\theta_2}{n_0}}\tilde{W}\tilde{X}
\right)^\top\right]
\end{equation}
where $\tilde{Z}\in\mathbb{R}^{n_1\times m}$, $\tilde{W}\in\mathbb{R}^{n_1\times n_0}$ and $\tilde{X}\in\mathbb{R}^{n_0\times m}$ are independent matrices filled with i.i.d Gaussian random variables.

We start by giving a result of convergence to the edge of the spectrum for a specific class of activation functions.  
\begin{theorem}\label{theo:edge}Let $\lambda_1$ be the largest eigenvalue of $M.$ 
Then, if $f$ is such that $\theta_3(f)=0$
\[
\lambda_{1}\xrightarrow[n_1\rightarrow\infty]{} \bf{u_+}\quad\text{in probability}.
\]
\end{theorem}
\begin{remark}This first result has to be compared with \cite{MontanariFan} where they show that there is no outlier for odd functions for kernel matrices. Our assumptions on the activation function $f$ are much stronger but our result holds true for a wider class of distributions than $W$ and $X$ Gaussian. We also have the weaker condition that $\theta_3(f)=0$ instead of $f$ being odd. This universal result is also of interest as it should be compared with Theorems \ref{theo:sep_easy} and \ref{theo:edge2} where we show that the behavior of the largest is distribution dependent. We expect however that one can show Theorem \ref{theo:edge} under the assumptions of \cite{MontanariFan} for Gaussian matrices $W$ and $X$ using stronger concentration bounds.
\end{remark} 

We now state the result for the largest eigenvalue in the general case. 
While we cannot describe the actual position and phase transition in the general case of parameters, we are able to relate it to a linear model for which the presence of outliers could be analyzed.

\begin{theorem}\label{theo:edge2}The largest eigenvalue $\lambda_1$ of $M$ has the same asymptotic behavior as that of the linear random matrix
	\[
	M_{\mathrm{Lin}}=\frac{1}{m}\left(\sqrt{\theta_1-\theta_2}\tilde{Z} +\sqrt{\frac{\theta_2}{n_0}} {\tilde{W}\tilde{X}}+J_2\right)
	\left(\sqrt{\theta_1-\theta_2}\tilde{Z} +\sqrt{\frac{\theta_2}{n_0}} {\tilde{W}\tilde{X}}+J_2\right)^\top
	\]
	Hereabove, $\tilde{Z}\in\R^{n_1\times m}, \tilde{W}\in\R^{n_1\times n_0}, $ and $\tilde{X}\in\R^{n_0\times m}$ are independent random matrices with i.i.d. $\mathcal{N}(0,1)$ entries
	and the matrix $J_2=J_2(f,W,X,n_0)\in\R^{n_1\times m}$ is a rank 2 block matrix with blocks of size $n_1/2\times m/2$: 
	\[
%	J_2=\begin{pmatrix} \sum_{i=1}^{n_0}\frac{X_{i,1}^2-\sigma_x^2}{n_0\sigma_x^2}, &\ldots ,& \sum_{i=1}^{n_0}\frac{X_{i,m/2}^2-\sigma_x^2}{n_0\sigma_x^2} &0&\ldots, &0\cr
%		\sum_{i=1}^{n_0}\frac{X_{i,1}^2-\sigma_x^2}{n_0\sigma_x^2}, &\ldots ,& \sum_{i=1}^{n_0}\frac{X_{i,m/2}^2-\sigma_x^2}{n_0\sigma_x^2} &0&\ldots, &0\cr
%		
%		\ldots\cr
%		\sum_{i=1}^{n_0}\frac{X_{i,1}^2-\sigma_x^2}{n_0\sigma_x^2}, &\ldots ,& \sum_{i=1}^{n_0}\frac{X_{i,m/2}^2-\sigma_x^2}{n_0\sigma_x^2} &0&\ldots, &0\cr
%		0&\ldots, &0 &\sum_{i=1}^{n_0}\frac{W_{1,i}^2-\sigma_w^2}{n_0\sigma_w^2}, &\ldots ,& \sum_{i=1}^{n_0}\frac{W_{m/2, i}^2-\sigma_w^2}{n_0\sigma_w^2} \cr
%		0&\ldots, &0 &\sum_{i=1}^{n_0}\frac{W_{1,i}^2-\sigma_w^2}{n_0\sigma_w^2}, &\ldots ,& \sum_{i=1}^{n_0}\frac{W_{m/2, i}^2-\sigma_w^2}{n_0\sigma_w^2} \cr
%		\ldots\cr
%		0&\ldots, &0 &\sum_{i=1}^{n_0}\frac{W_{1,i}^2-\sigma_w^2}{n_0\sigma_w^2}, &\ldots ,& \sum_{i=1}^{n_0}\frac{W_{m/2, i}^2-\sigma_w^2}{n_0\sigma_w^2} 
%	\end{pmatrix}
	J_2 = \sqrt{\frac{4\theta_3(f)}{n_0}}
		\begin{pmatrix}
			\sqrt{\kappa_w}J & 0\\
			0 & \sqrt{\kappa_x}J
		\end{pmatrix}
	. \]
	where $J\in\R^{n_1/2\times m/2}$ is the matrix whose all entries are equal to 1.
\end{theorem}
%The above result actually encompasses the two previous Theorems \ref{theo:edge} and \ref{theo:sep_easy}. Indeed, if we consider the matrix given by $4\theta_3m^{-1}J_2J_2^\top$, the top left block are given by
%\[
%\frac{4\theta_3}{mn_0^2\sigma_x^4}\sum_{k=1}^{m/2}\sum_{i,j=1}^{n_0}(X_{i,k}^2-\sigma_x^2)(X_{j,k}^2-\sigma_x^2)
%=
%\left(\frac{2\theta_3 \kappa_x}{\psi n_1}+\O{\frac{1}{n_1^{3/2}}}\right)J_{n_1/2\times n_1/2}
%\] 
%where $J$ is the rank 1 matrix full of 1's and $\kappa_x$ is defined in \eqref{kappa}. Note that we used the law of large numbers, the central limint theorem and the fact that $\psi n_1\sim n_0$ so that we recover the same structure as in Theorem \ref{theo:sep_easy}. The same reasoning can be done for the bottom right block.

We first note that the theorem gives the asymptotic position of the largest eigenvalue but we do not believe that the corresponding eigenvector is asymptotically given by the extreme eigenvector of $M_{\mathrm{Lin}}.$ While it would be of interest to understand the eigenspace structure and the possible impact of the distribution of $W$ and$X$ or the activation function $f$, this is beyond the scope of this paper. This result also gives a candidate for the position of the second largest eigenvalue since $J_2$ is a rank $2$ matrix and could create two outliers in the information-plus-noise matrix $M_{\mathrm{Lin}}$. In the proof we show that 
\[
\mathbb{E}\text{Tr}M^q=(1+o(1))\mathbb{E}\text{Tr}M_{\mathrm{Lin}}^q
\]
 which would suggest that the eigenvalues exiting the support of the limiting e.e.d. are the same as those for  the modified matrix $M_{\mathrm{Lin}}$.

\begin{remark}In the regime considered $q\sim (\ln n_0)^{1+\alpha}$ for some $\alpha>0$ small, one can show that 
	\begin{eqnarray}\label{2.7}&\mathbb{E}\left (\text{Tr} M^q\right) = &\cr
		&&(1+o(1))\sum_{I_i,I_j=0}^q\sum_{b=0}^{I_i+I_j}\sum_{L=4}^{q-b}{\mathcal{A}}(q,I_i,I_j,b; L)\theta_1(f)^b\theta_2(f)^{q-b-L}\psi^{I_i+1-q}\phi^{I_j}\theta_3(f)^L \cr
		&&\times \left ( (\kurt(W)-1)^L+ (\kurt(X)-1)^L\right )\cr
		&&+(1+o(1))\sum_{I_i,I_j=0}^q\sum_{b=0}^{I_i+I_j+1}{\mathcal{A}}(q,I_i,I_j,b)\theta_1(f)^b\theta_2(f)^{q-b}\psi^{I_i+1-q}\phi^{I_j}.\end{eqnarray}
	Hereabove  $\mathcal{A}(q,I_i,I_j,b; L)$ counts the number of admissible graphs with $b$ simple cycles and one marked long cycle of length $L$ which is defined later.
	The first sum is as the moment of the limiting e.e.d. except that one long cycle of length $L$ is assigned the weight $\theta_3(f)^L$ times some factor depending on $W$ or $X$.
	The second sum can be shown to be equal to $\int x^q d\mu =O(1) u_+^q$. 
\end{remark}

In the particular case where $\theta_2=0$, we obtain an easier linear model since the product $\tilde{W}\tilde{X}$ disappear from the expression. In particular, we prove an equivalent theorem which consists in a rank 1 perturbation of an i.i.d matrix. This formulation allows us to give an explicit formula for the asymptotic position of the largest eigenvalue since finite rank perturbation of the Marchenko--Pastur are well understood.
\begin{theorem}\label{theo:sep_easy}
If $f$ is such that $\theta_2(f)=0$, the largest eigenvalue of $M$ has the same limit as the largest eigenvalue of the following information-plus-noise matrix 
\begin{equation}\label{capmodel}
	\widetilde{M}_{\mathrm{Lin}}
		=\frac{1}{m}\left (\sqrt{\theta_1} \tilde{Z} +\sqrt{\frac{\theta_3\kappa}{n_0}} J\right)\left ( \sqrt{\theta_1} \tilde{Z} +\sqrt{\frac{\theta_3\kappa}{n_0}} J\right)^\top,\end{equation}
where $\kappa=\max (\kurt(W)-1, \kurt(X)-1 )$, $\tilde{Z}$ is a $n_1\times m$ random matrix with i.i.d. $\mathcal{N}(0,1)$ entries and $J$ is the $n_1\times m$ matrix whose all entries are equal to $1$. In particular, we have that 
\[
\lambda_1 \xrightarrow[n_1\to \infty]{}
\left\{
	\begin{array}{ll}
		\displaystyle{\theta_1 \frac{(1+\alpha)(\gamma + \alpha)}{\alpha}} & \displaystyle{\text{if } \alpha>\sqrt{\gamma}\theta_1},\\[1.75ex]
		\left(1+\sqrt{\gamma}\right)^2 &\text{otherwise}.
	\end{array}
\right.
\quad\text{ with }\alpha\coloneqq\frac{\theta_3\kappa}{\psi}
\]
where we recall that $\gamma=\frac{\phi}{\psi}$.
\end{theorem}
 The second part of the theorem which describes the separation of outliers has been proved  in \cite{Capitaine} for the precise model (\ref{capmodel}). The first observation is that the presence of outliers and its position is nonuniversal as it depends on the distribution of $W$ and $X$ through the parameter $\kappa$. This is illustrated in Figure \ref{fig:nonuniv}.

\begin{figure}[ht!]
	\centering
	\begin{minipage}{.30 \linewidth}
		\includegraphics[width=\linewidth]{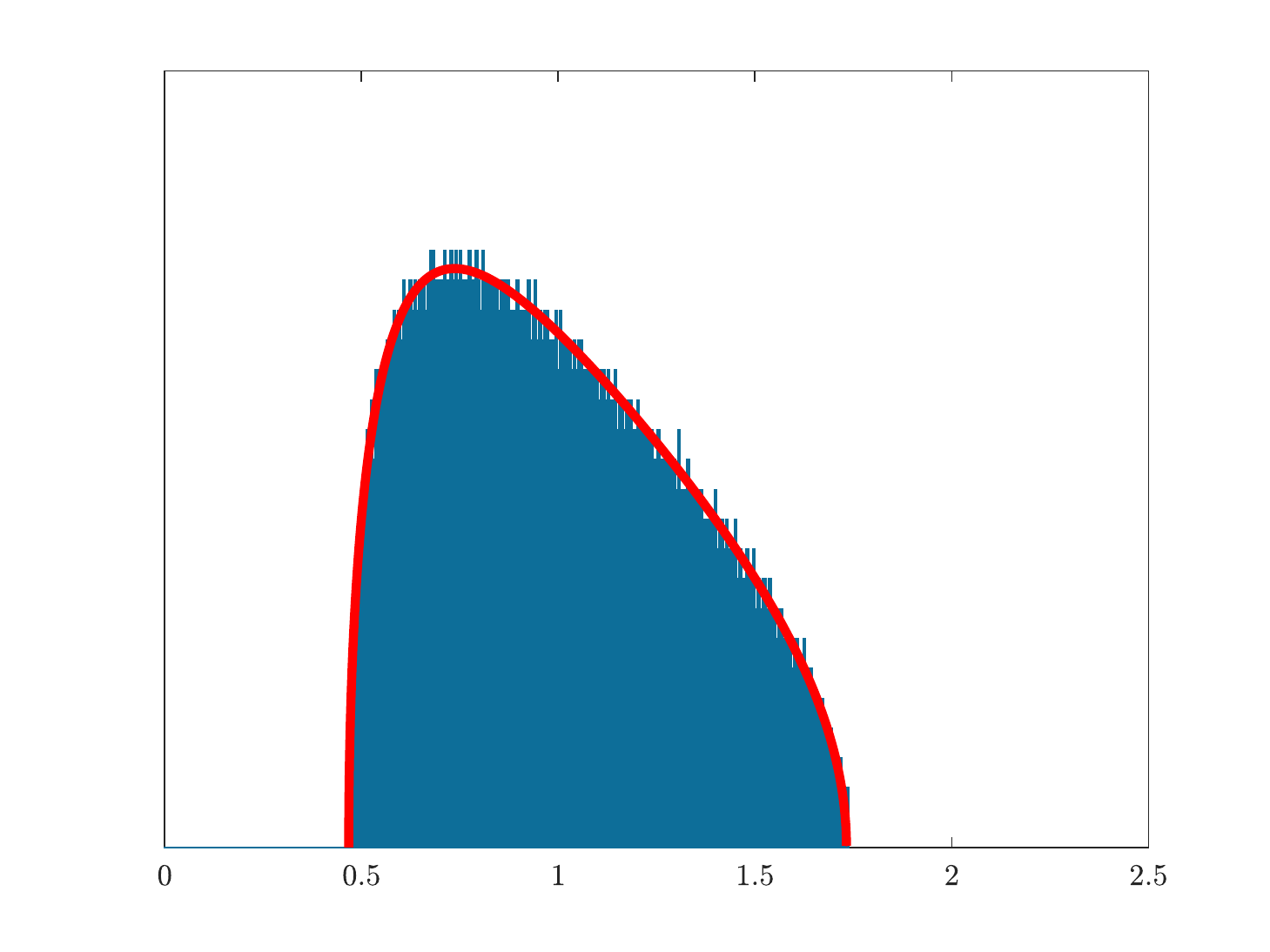}
	\end{minipage}
	\begin{minipage}{.30 \linewidth}
		\includegraphics[width=\linewidth]{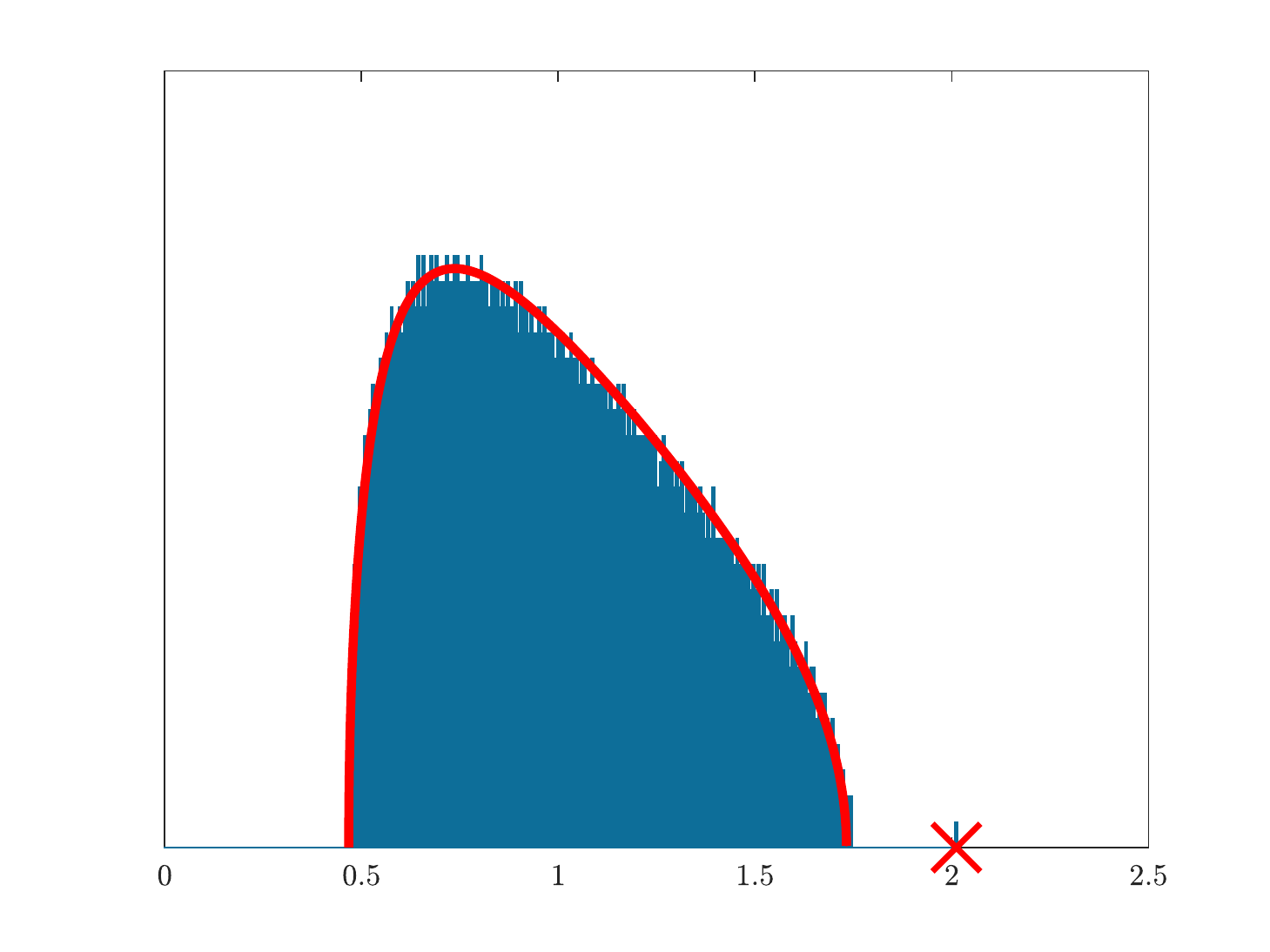}
	\end{minipage}
	\begin{minipage}{.30 \linewidth}
		\includegraphics[width=\linewidth]{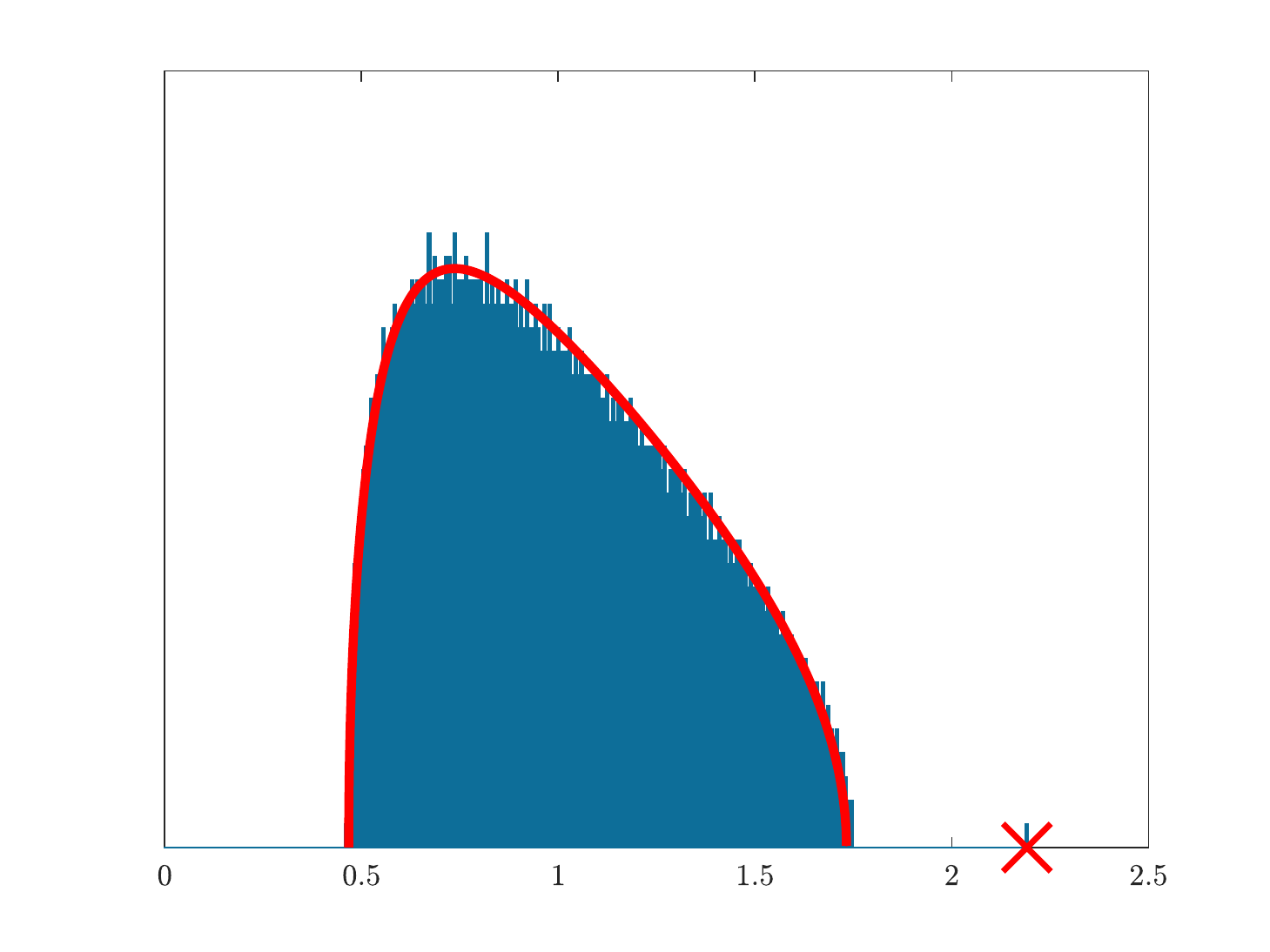}
	\end{minipage}
	\caption{Here $X_{ij}\sim \mathrm{Ber}$ with $\mathrm{Ber}$ being $\pm 1$ with probability 1/2, $f(x)=(x^2-1)/\sqrt{2}$, $\phi=.1$ and $\psi=1$. We then change the distribution of $W$. In the left picture, we take $W_{ij}\sim \mathrm{Ber}$ : there is no outlier as $\kappa=0$; in the center, we choose $W_{i,j}\sim \frac{1}{4}\delta_{\mathrm{Ber}}+\frac{3}{4}\delta_{\mathcal{N}(0,1)}$ with an outlier appearing as $\kappa>0$; finally, on the right, we consider $W_{ij}\sim \mathcal{N}(0,1)$ and the outlier is further away from the bulk of the spectrum as $\kappa$ increases.}
	\label{fig:nonuniv}
\end{figure}

Interestingly, we also see that the behavior of the largest eigenvalue depends on the activation function in a different way than the e.e.d as it depends also on $\theta_3$. This is illustrated in Figure \ref{fig:theta_3}.

\begin{figure}[ht!]
	\centering
	\begin{minipage}{.30 \linewidth}
		\includegraphics[width=\linewidth]{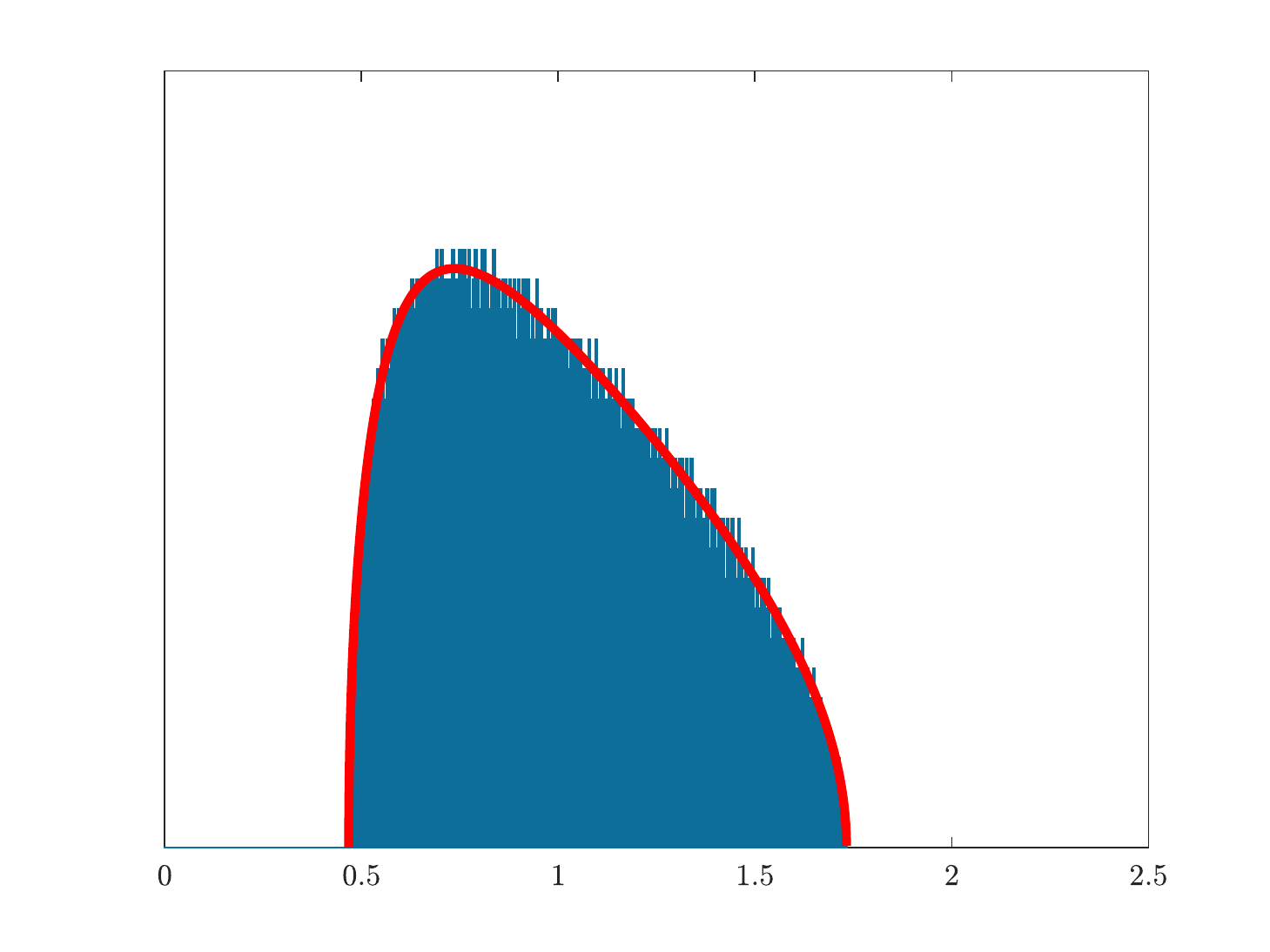}
	\end{minipage}
	\begin{minipage}{.30 \linewidth}
		\includegraphics[width=\linewidth]{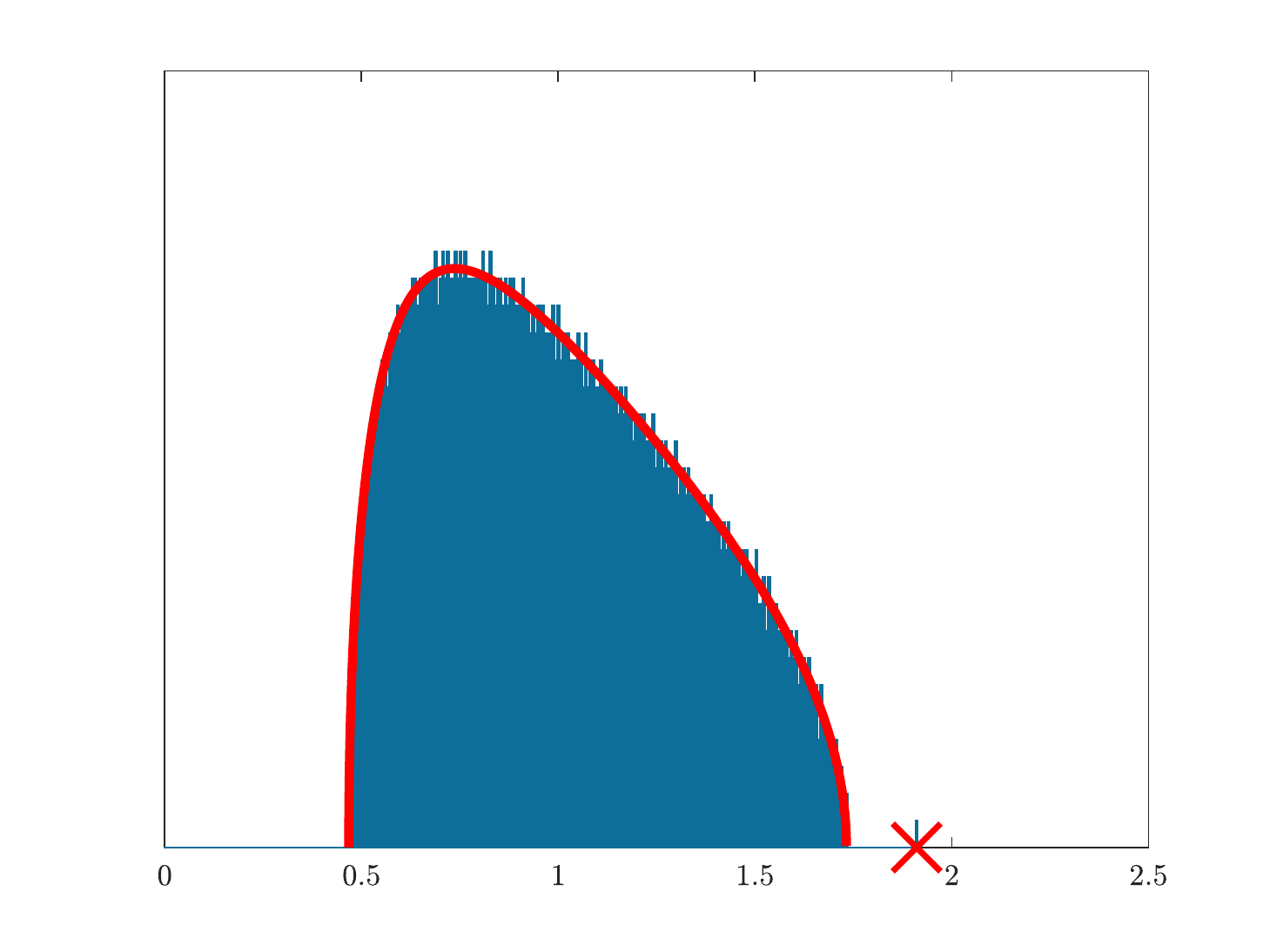}
	\end{minipage}
	\begin{minipage}{.30 \linewidth}
		\includegraphics[width=\linewidth]{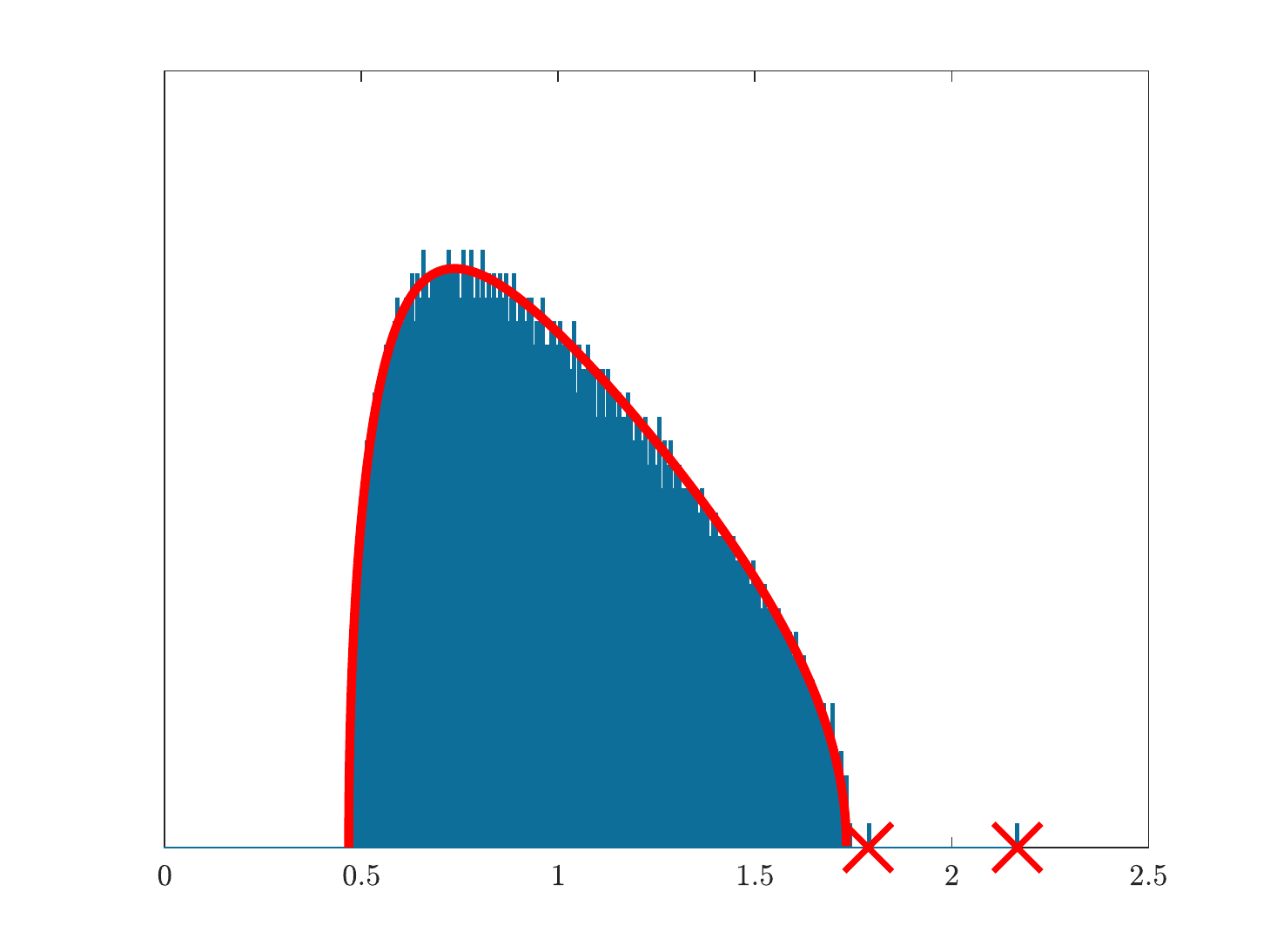}
	\end{minipage}
	\caption{Here $W$ has standard Gaussian entries, the entries of $X$ are distributed as $\frac{1}{2}\delta_{\mathrm{Ber}}+\frac{1}{2}\delta_{\mathcal{N}(0,1)}$, $\phi=.1$ and $\psi=1$ where $\mathrm{Ber}$ is $\pm 1$ with probability $1/2$. We then consider the one parameter family of activation functions $f_{\alpha}(x)=\frac{\cos(\alpha x)-\mathrm{e}^{-\alpha^2/2}}{\sqrt{\mathrm{e}^{-\alpha^2}(\cosh(\alpha^2)-1)}}$ for which we have $\theta_1(f_\alpha)=1$ and $\theta_2(f_\alpha)=0$ regardless of $\alpha$ while $\theta_3(f_\alpha)$ varies depending on $\alpha$. We see that as $\alpha$ decreases ($\alpha=2$ (left), $\alpha=1.5$ (center), $\alpha=.8$ (right)) and $\theta_3(f)$ increases  some outliers are appearing while the overal e.e.d. remains the Marchenko--Pastur distribution (in red).}
	\label{fig:theta_3}
\end{figure}

Finally, we see that the position of the largest eigenvalue depends also on the \emph{architecture} of the neural network (in terms of the different dimensions since we consider only a single layered network here) as it depends on the parameters $\psi$ (or $\phi$ depending on the normalization the matrix $J$). In particular, if we consider the case $\theta_2=0$, we see that the dimension $n_0$ does not appear in the asymptotic eigenvalue distribution since it is given by the Marchenko--Pastur distribution of shape $\gamma=\frac{\phi}{\psi}=\lim \frac{n_1}{m}$. However, every dimension involved in the layer appears when considering both the empirical eigenvalue distribution and the largest eigenvalue.  Thus the choice of dimension size is of importance for the largest eigenvalue as illustrated in Figure \ref{fig:archi}.

\begin{figure}[ht!]
	\centering
	\begin{minipage}{.30 \linewidth}
		\includegraphics[width=\linewidth]{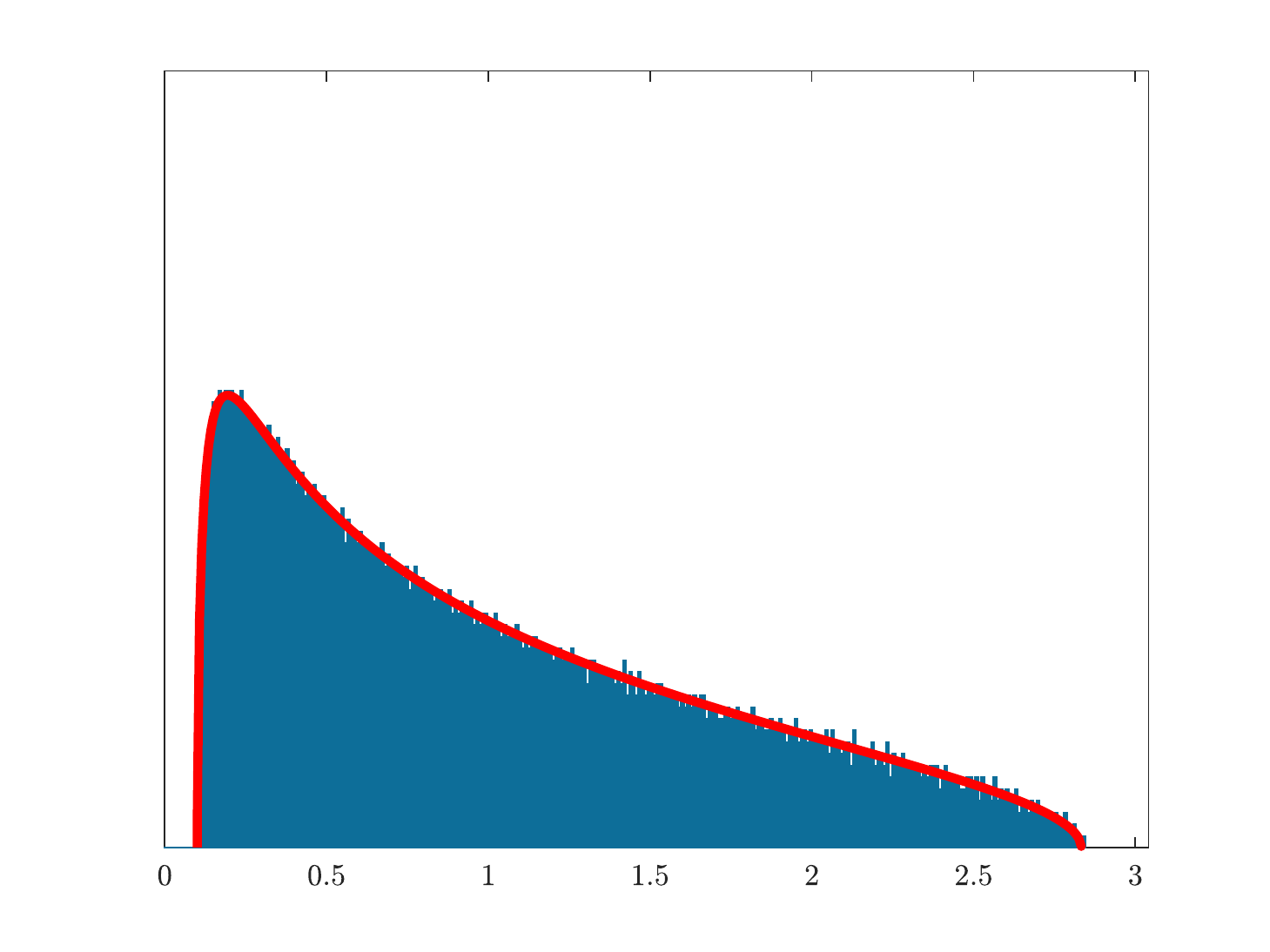}
	\end{minipage}
	\begin{minipage}{.30 \linewidth}
		\includegraphics[width=\linewidth]{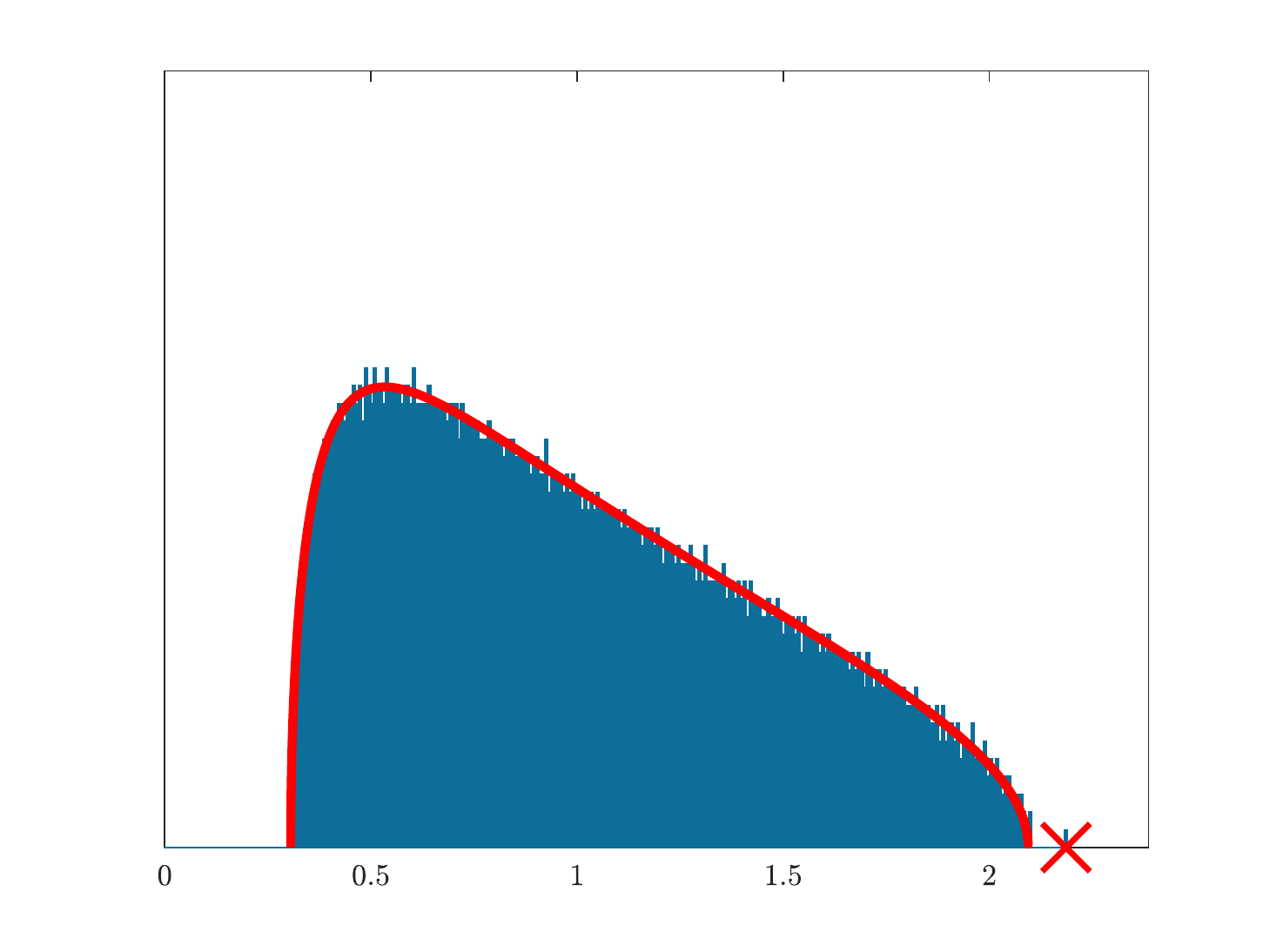}
	\end{minipage}
	\begin{minipage}{.30 \linewidth}
		\includegraphics[width=\linewidth]{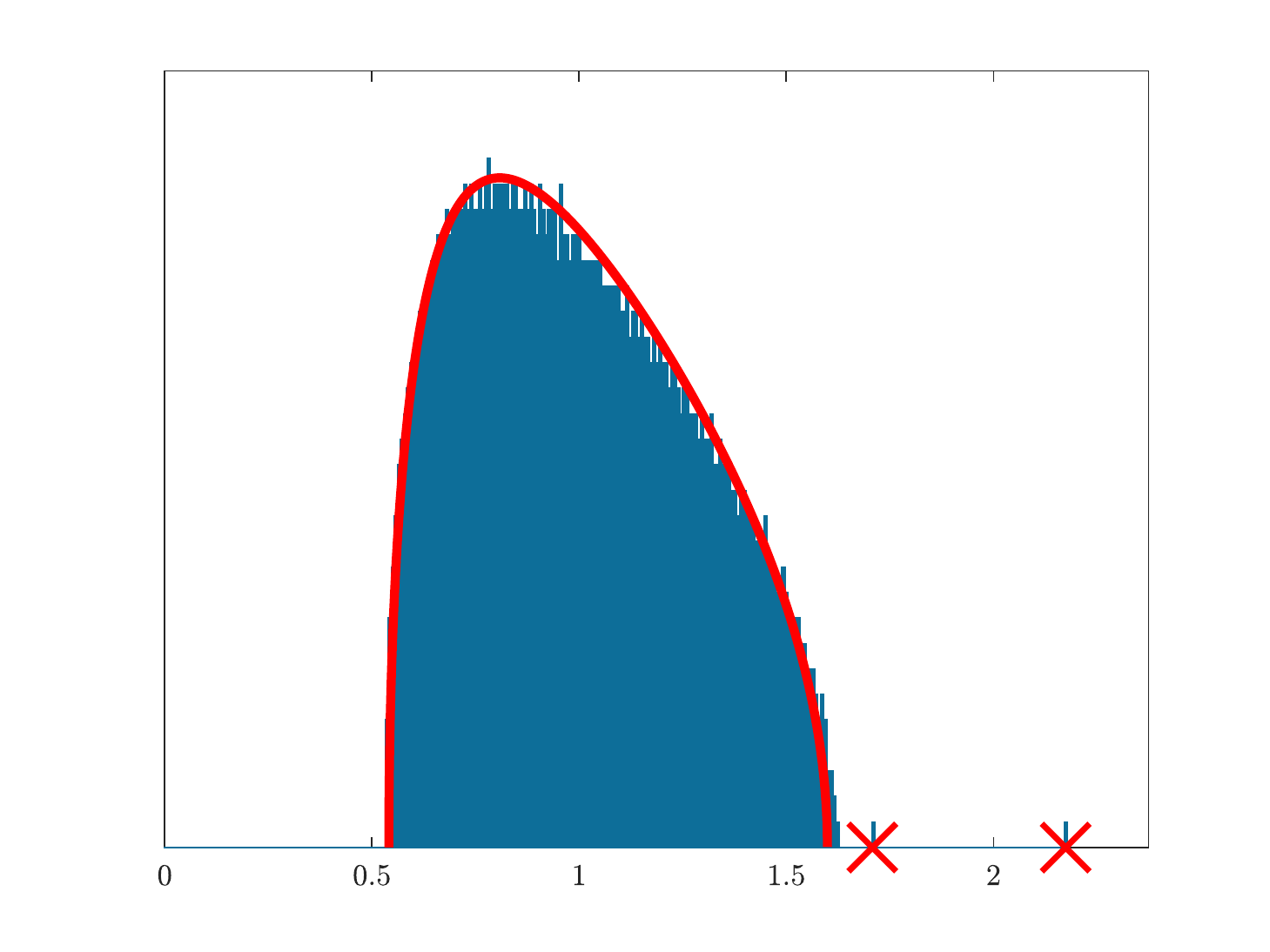}
	\end{minipage}
	\caption{ $W$ has standard Gaussian entries, the entries of $X$ are distributed as $\frac{1}{2}\delta_{\mathrm{Ber}}+\frac{1}{2}\delta_{\mathcal{N}(0,1)}$, and $f(x)=(x^2-1)/\sqrt{2}$. On the left, we take $\phi=.7$, $\psi=1.5$ with no outliers. For the middle picture, we have $\phi=.3$ and $\psi=1.5$ with one outlier appearing. Finally, on the right, we consider $\phi=0.07$ and $\psi=1$ with two outliers. The e.e.d. is the Marchenko--Pastur distribution (in red) but with different shapes given by $\phi/\psi$.}
	\label{fig:archi}
\end{figure}
 
\vspace{\baselineskip}
\begin{remark}We cannot fully describe the behavior of the second largest eigenvalue unfortunately as we can see that a second outlier can appear in Figures \ref{fig:theta_3} or \ref{fig:archi}. If  $\kurt(W)-1$ and/or $\kurt(X)-1$ is null, then no/at most one eigenvalue separates. This is e.g. the case when $W$ and/or $X$ has i.i.d. Bernoulli entries which is consistent with the  observation from \cite{LouartLiaoCouillet}.  
\end{remark}

In the general case where $\theta_2\neq 0$ and $\theta_3\neq 0$, the position of the largest eigenvalue of $M_{\mathrm{Lin}}$ is not known to our knowledge. However, one can instead take a related model which should give the same position, it is a model studied in \cite{benaychsingular}. Consider the matrix
\[
\hat{M}_{\mathrm{Lin}}=\frac{1}{m}\left(\sqrt{\theta_1-\theta_2}\tilde{Z} +\sqrt{\frac{\theta_2}{n_0}} {\tilde{W}\tilde{X}}+\sqrt{\frac{\theta_3\kappa}{n_0}}\hat{J}\right)
\left(\sqrt{\theta_1-\theta_2}\tilde{Z} +\sqrt{\frac{\theta_2}{n_0}} {\tilde{W}\tilde{X}}+\sqrt{\frac{\theta_3\kappa}{n_0}}\hat{J}\right)^\top
\]
where $\hat{J}=\mathbf{u} \mathbf{v}^\top$ where $\mathbf{u}$ and $\mathbf{v}$ are column vectors with i.i.d. standard Gaussian random variables. One can see that $M_{\mathrm{Lin}}$ is the same model where we change the rank 1 perturbation by $J=\mathbf{u} \mathbf{v}^\top$ where $\mathbf{u}$ and $\mathbf{v}$ are vectors whose entries are 1. However, we expect the position of the largest eigenvalue to be the same between $M_{\mathrm{Lin}}$ and $\hat{M}_{\mathrm{Lin}}.$ Define the $D$-transform of the measure $\mu$ defined in Theorem \ref{theo:bepe}, for $z>\sqrt{\mathbf{u}_+}$ (the right edge of the support of $\mu$),
\[
D_\mu(z)
\coloneqq
\left(
	\int \frac{z}{z^2-x}\d \mu(x)
\right)
\left(
	\gamma \int \frac{z}{z^2-x}\d \mu(x)
	+
	\frac{1-\gamma}{z}
\right).
\]
While $D_\mu$ does not have an explicit form, it can be written in terms of the Stieltjes transform of $\mu$ which follows an explicit quartic self-consistent equation \cite{BePe}*{Theorem 2.3}. The position of the largest eigenvalue of $\hat{M}_{\mathrm{Lin}}$ is given by,
\[
\lambda_1 \xrightarrow[n_1\to\infty]{}
\left\{
	\begin{array}{ll}
		D^{-1}_\mu(\alpha^{-1}) &\displaystyle{\text{if }\alpha > \lim_{z\downarrow \sqrt{\mathbf{u}_+}}(D_\mu(z))^{-1}}, \\[1.75ex]
		\mathbf{u}_+ &\text{otherwise}
	\end{array}
\right.
\]
where $D_\mu^{-1}$ denotes the functional inverse of $D_\mu$. Thus we see that if $\theta_3\kappa \psi^{-1}$ is large enough (depending on $\theta_1$, $\theta_2,$ $\phi,$ and $\psi$), we have separation of the largest eigenvalue from the bulk of the spectrum.
\paragraph{}
Our strategy for the proof is the following: 

\paragraph{Polynomial approximation.}First we approximate the function $f$ by a Taylor polynomial $P_k$, then replacing $Y$ with $Y_k$. The degree of the latter polynomial has to be large enough so that the largest eigenvalues of $Y$ and that of $Y_k$ are asymptotically the same. We use the same notations as in \cite{BePe}: define
\begin{equation}\label{eq:defpoly}
P_k(x):=\sum_{j=1}^k f^{(j)}(0)\frac{x^j-j!!}{j!}=\sum_{j=0}^k f^{(j)}(0)\frac{x^j}{j!}-a_{k}
\quad\text{with}\quad
a_k=\sum_{j=0}^k f^{(j)}(0)\frac{j!!}{j!}
\end{equation}
with the convention that $j!!=0$ for $j$ odd and $0!!=1$. This choice ensures that the polynomial is centered with respect to the Gaussian distribution.  
Now, we compare the Hermitized version of the matrix $M$ (up to finite rank modification), and define
\begin{equation}\label{eq:defyk}
\begin{gathered}
Y^{(a_k)}=f\left(\frac{WX}{\sqrt{n_0}}\right)-a_k,
\quad
Y_k=P_k\left(\frac{WX}{\sqrt{n_0}}\right),
\quad M_k:=\frac{1}{m}Y_kY_k^\top\\
\mathcal{E}=\frac{1}{\sqrt{m}}
\begin{pmatrix}
	0 &	Y^{(a_{k-1})}-Y_k\\
	\left(Y^{(a_{k-1})}-Y_k\right)^\top & 0
\end{pmatrix}.
\end{gathered}
\end{equation}

\paragraph{Result when $\theta_3(f)=0$.} This is an application of the moment analysis from \cite{BePe}.  We study $\mathbb{E} \text{Tr}M_k^q$ for large integers $q$ and show that these moments are still appropriately approximated by $\int x^q d\mu$. Indeed, the convergence of moments hold up to $q\sim (\ln n_1)^{1+\alpha}$ when $\theta_3(f)=0$ which allows us to give an upper bound on the largest eigenvalue based on the fact that $\Tr M_k^q \geqslant \lambda_1^q$ combined with Markov's inequality. The matching lower bound simply comes from the convergence of the e.e.d. to $\mu$.

\paragraph{Result in the general case.} We also need to push the analysis of moments from \cite{BePe} to high enough moments. However, in the general case, some matchings and graphs which are not contributing in the finite moment case starts to possibly contribute. These new matchings involve $\theta_3$ as well as $\kappa$. We then relate these moments to those of the e.e.d. of information-plus-noise random matrices.

The main difference between activation functions such that $\theta_3(f)=0$ or not comes from the combinatorics of the spectral moments of $M$. To investigate the behavior of the largest eigenvalues, it is now standard to use the high order moments of the e.e.d as first seen in \cite{furedi1981eigenvalues}. Indeed it is reasonable to expect that as $q$ grows to infinity, in an appropriate way, $\E \Tr M^q\simeq \lambda_1^q C$ for some constant $C$.
 For an activation function such that $\theta_3(f)=0$, the same graphs contributing to finite spectral moments are contributing for large moments.
 
 If we have $\theta_2(f)=0$, only a fraction of such graphs, namely the trees of simple cycles, contribute to the limiting spectral moments for fixed integer $q$. However, for large $q$ and thus regarding the behavior of the largest eigenvalue, it may hold true that some other admissible graphs contribute. Actually we show that a non negligible contribution comes from admissible graphs with a single long cycle and simple cycles for the rest. Thus in this case, one may observe a contribution which comes from the largest eigenvalues and not related to the support of the limiting measure $\mu$. In the general case, all admissible graphs contribute but we obtain a different contribution than that for fixed integer $q$, since different \emph{matchings} on admissible graphs can contribute.

\section{Preliminary steps}
\subsection{Approximating $f$ with a polynomial}
In this subsection, we show that one can replace the activation function $f$ with its Taylor approximation $P_k$.  We here fix the degree $k$ of the approximation  so that the behavior of the largest eigenvalue for $M$ and $M_k$ are the same in the large dimension limit.

We now show that $f$  can be approximated by a polynomial $P_k$. One considers the centered Taylor--Lagrange approximation polynomial $P_k$ defined in \eqref{eq:defpoly} and consider also $Y^{(a_k)}$ and $Y_k$ defined in \eqref{eq:defyk}.
Define then 
\begin{equation}
	R_k:= \frac{1}{m}Y^{(a_k)}(Y^{(a_k)})^\top -\frac{1}{m}Y_kY_k^\top.
\end{equation}
Let us set \begin{equation}\label{eq:notyf}
	\mathcal{A}_{n_1}(\delta_1)=\bigcap_{1\leqslant i\leqslant n_1}\bigcap_{1\leqslant j\leqslant m}
	\left\{
	\left\vert
	\left(
	\frac{WX}{\sqrt{n_0}}
	\right)_{ij}
	\right\vert
	\leqslant (\log n_1)^{1/2+\delta_1}
	\right\}.
\end{equation}

The spectral radius of $R_k$ can be bounded from above on the very high probability event $\mathcal{A}_{n_1}(\delta_1)$ defined in \eqref{eq:notyf} by
\begin{equation}\label{eq:boundspecradius}
	\rho({R_{k-1}})
	\leqslant
	C_f\sqrt{m}\frac{(\log n_1)^{(1+c_f)(1/2+\delta_1)k}}{k!}.
\end{equation}
The above goes to 0 as $n_1$ tends to infinity provided that $k \geqslant c_0 \frac{\log n_1}{\log \log n_1}$ for a constant $c_0>1$. From now on, we fix such a degree $k$ for the approximation. Then, the largest eigenvalue of $Y^{(a_k)}(Y^{(a_k)})^\top/m$ will be suitably approximated by that of $M_k=Y_kY_k^\top/m.$

Now, in order to control the largest eigenvalue of $M=YY^\top/m$, we note that $Y$ is a rank one perturbation of $Y^{(a_k)}$. Such a perturbation can possibly change the behavior of the largest eigenvalue but the perturbation here is small, indeed as the activation function $f$ has a zero Gaussian mean, we have that
	\[
	\sum_{j=0}^\infty f^{(j)}(0)\frac{j!!}{j!}=
	\mathds{E}\left[
	\sum_{j=0}^\infty f^{(j)}(0)\frac{\mathcal{N}^j}{j!}
	\right]
	=
	\mathds{E}\left[
	f(\mathcal{N})
	\right]
	=
	0,
	\]
	with $\mathcal{N}$ a standard Gaussian random variable. Thus, we have that there exists a constant $C>0$ such that 
	\[
	\vert 
	a_{k}
	\vert
	=
	\left\vert
	\sum_{j=0}^kf^{(j)}(0)\frac{j!!}{j!}
	\right\vert
	=
	\left\vert
	\sum_{j=k}^\infty f^{(j)}(0)\frac{j!!}{j!}
	\right\vert
	\leqslant
	\frac{C^{(k-1)}}{(k-1)^{(k-1)/2}}.
	\]
	By Proposition \ref{prop:highmoment}, $k$ can be as large as $c_0\frac{\log n_1}{\log \log n_1}$ for any $c_0>0$. In this case we obtain that for any $\varepsilon>0$ we have $a_k=\mathcal{O}(n^{-c_0/2+\varepsilon}).$ Now, we use the Hoffman--Wielandt inequality for singular values to finish, indeed we have
	\[
	\left\vert
	\sqrt{\lambda_1(Y^{(a_k)})}
	-
	\sqrt{\lambda_1(Y)}
	\right\vert
	\leqslant
	\sqrt{
		\sum_{i=1}^{n_1}
		\left(
		\sqrt{\lambda_i(Y^{(a_k)})}-\sqrt{\lambda_i(Y)}
		\right)^2
	}
	\leqslant
	\left\Vert
	Y^{(a_k)}-Y
	\right\Vert,
	\]
	with $\Vert A\Vert^2 = \Tr AA^\top$. One then has  that
	$
	\left\Vert
	Y^{(a_k)}-Y
	\right\Vert
	=
	\sqrt{a_k^2mn_1}=\mathcal{O}\left(
	\frac{C}{n_1^{c_0/2-2-\varepsilon}}
	\right).$
	We finally obtain the result by taking $c_0>4+2\varepsilon$.
	
\subsection{Reminders of combinatorics in \cite{BePe}}	

In this subsection, we recall the combinatorics of the spectral moments of $M_k$ from \cite{BePe} and consider $\mathbb{E}\text{Tr}M_k^q$ for some fixed integer $q$ (not depending on $n$). This subsection will be the basis for the combinatorics we develop as $q$ increases with the dimension.
 
By using the previous subsection, we can assume that the activation function is a polynomial. For ease, we here assume that the activation function is a monomial of degree $k$ and assume for ease $f(x)=x^k.$.

The expected value 
$$\E \Tr \left (\frac{YY^\top}{m} \right)^q , \text{ where }Y_{ij}=\left(\frac{WX}{\sqrt{n_0}}\right)_{ij}^k, 1\leq i \leq n_1, 1\leq j \leq m$$ can be encoded into the contribution of some graphs. Indeed 
one has that 
\begin{eqnarray}
	&&\frac{1}{m^q}\E\text{Tr}(YY^\top)^q=\sum_{i_1, \ldots, i_q=1}^{n_1}\sum_{j_1, \ldots, j_q=1}^m \E \prod_{r=1}^q Y_{i_r j_r}Y_{i_{r+1}j_r},
\end{eqnarray}
where we use the convention that $i_{q+1}=i_1.$ 
Each summand is encoded into a red bipartite graph: one simply draws an edge between the vertices $i_r$ and $j_r$ or $j_r$ and $i_{r+1}$. The graph is bipartite due to the different possible labeling of $i$-indices and $j$-indices. 
Now there may be some coincidences among the $i$ or $j$- indices. This is the place where admissible graphs arise whose definition we recall from \cite{BePe}.
\begin{definition}\label{def:graph} Let $q\geq 1$ be a given integer.
	A coincidence graph  is a connected graph built up from the simple (bipartite) cycle  of  vertices labeled $i_1, j_1, i_2, \ldots, i_q, j_q$ (in order) by  identifying some $i$-indices respectively and $j$-indices respectively. Such a graph is admissible if the formed cycles are joined to another by at most a common vertex and each edge belongs to a unique cycle. 
\end{definition}

We denote $\A(q,I_i,I_j,b)$ the number of admissible graphs with $2q$ edges, with $I_i$ identifications between $i$-indices, $I_j$ identifications between $j$-indices and $b$ cycles of length 2.

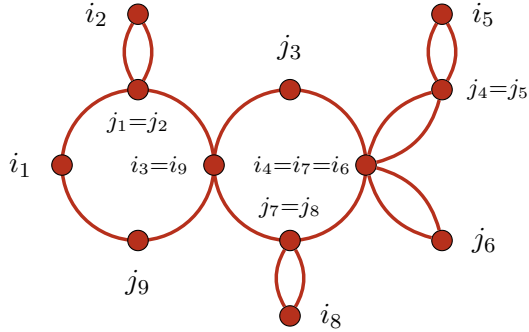
\begin{figure}[ht!]
	\centering
	\begin{tikzpicture}
		\draw[-,BrickRed,line width = .12em] (0,0) circle (1cm);
		\draw[-,BrickRed,line width = .12em] (2,0) circle (1cm);
		\node[fill=BrickRed,inner sep = 0pt, minimum size=.27cm, label=left:\fontsize{9}{9}$i_3{=}i_9$] (2) at (1,0) {};
		\node[fill=BrickRed,inner sep = 0pt, minimum size=.27cm,label=left:$i_1$] at (-1,0) {};
		\node[fill=BrickRed,inner sep = 0pt, minimum size=.27cm,label={[label distance = -.8em]below:\fontsize{9}{9}$j_1{=}j_2$}] (1) at (0,1) {};
		\node[fill=BrickRed,inner sep = 0pt, minimum size=.27cm, label=below:$j_9$] at (0,-1) {};
		\node[fill=BrickRed,inner sep = 0pt, minimum size=.27cm,label=left:$i_2$] (6) at (0,2) {};
		\node[fill=BrickRed,inner sep = 0pt, minimum size=.27cm, label=above:$j_3$] at (2,1) {};
		\node[fill=BrickRed,inner sep = 0pt, minimum size=.27cm,label={[label distance = -.8em]above:\fontsize{9}{9}$j_7{=}j_8$}] (5) at (2,-1) {};
		\node[fill=BrickRed,inner sep = 0pt, minimum size=.27cm,label={[label distance = -.3em]left:\fontsize{9}{9}$i_4{=}i_7{=}i_6$}] (3) at (3,0) {};
		\node[fill=BrickRed,inner sep = 0pt, minimum size=.27cm, label=right:$i_8$] (9) at (2,-2) {};
		\node[fill=BrickRed,inner sep = 0pt, minimum size=.27cm, label=right:\fontsize{9}{9}$j_4{=}j_5$] (4) at (4,1) {};
		\node[fill=BrickRed,inner sep = 0pt, minimum size=.27cm, label=right:$j_6$] (8) at (4,-1) {};
		\node[fill=BrickRed,inner sep = 0pt, minimum size=.27cm, label=right:$i_5$] (7) at (4,2) {};
		
		\draw[-,BrickRed, line width=.12em] (1) edge[bend left] (6); 
		\draw[-,BrickRed, line width=.12em] (1) edge[bend right] (6); 
		
		\draw[-,BrickRed, line width=.12em] (5) edge[bend left] (9);
		\draw[-,BrickRed, line width=.12em] (5) edge[bend right] (9); 
		
		\draw[-,BrickRed, line width=.12em] (3) edge[bend left] (4); 
		\draw[-,BrickRed, line width=.12em] (3) edge[bend right] (4); 
		
		\draw[-,BrickRed, line width=.12em] (3) edge[bend left] (8); 
		\draw[-,BrickRed, line width=.12em] (3) edge[bend right] (8); 
		
		\draw[-,BrickRed, line width=.12em] (4) edge[bend left] (7); 
		\draw[-,BrickRed, line width=.12em] (4) edge[bend right] (7); 
	\end{tikzpicture}
	\caption{Example of an admissible graph with 7 cycles including 5 cycles of length 2, 3 $i-$identifications, 3 $j-$identifications belonging in $\mathcal{A}(9, 3,3,5)$.}
	\label{fig:admiss}
\end{figure}

By specifying the activation function,  we have,

\[
Y_{ij}=\left(\frac{WX}{\sqrt{n_0}}\right)^k=\frac{1}{n_0^{k/2}}\sum_{\ell_1,\dots,\ell_k=1}^{n_0}\prod_{p=1}^k W_{i\ell_p}X_{\ell_p j}.
\]
This can be encoded in the graph by adding $k$ $\emph{blue}$ vertices on each red edge of an admissible (or not) graph. A blue vertex labeled $l_p$ on the red edge $(i,j)$ stands for  $W_{i\ell_p}X_{\ell_p j}$ using this encoding. 
Thus, to get a non vanishing contribution to the spectral moment, each blue vertices have to be matched since $W$ and $X$ entries are independent and centered.
Indeed each entry $W_{i\ell_p}$ or $X_{\ell_p j}$ has to arise at least twice to give a non zero contribution. \\
We need to compute the leading contribution for $i$'s and $j$'s (which corresponds to admissible graphs defined above) and on top of it, perform matchings between blue vertices on each red edge  which we call a \emph{niche}. Note that considering an even, odd or general polynomial changes the parity of the number of blue vertices in each niche and thus the combinatorics for the blue matchings. However, changing the activation function does not change the red graph and the leading contribution is given by admissible graphs as shown in \cite{BePe}.

Consider the simple cycle of length $2q$ whose vertices are labeled alternatively with pairwise distinct $i-$indices from $\{1, \ldots , n_1\}$ and $j-$indices from $\{1, \ldots, m\}$. All the graphs contributing to the expectation are obtained from this simple cycle  by identifications of some vertices. These are the red graphs from \cite{BePe} as illustrated in Figure \ref{fig:admiss}. To be more precise, each red edge (also called a niche) is decorated by $k$ blue half edges which have to be matched (into pairs or cycles) so that no blue half edge is single. This is the necessary condition so that the contribution of the graph to the above expected value does not vanish. The combined contribution of such graphs
has been shown in \cite{BePe} to split into two parts: \begin{itemize}
\item the contribution of \emph{admissible} red graphs: the red graphs are cactus graphs, i.e. connected graphs where no edge belongs to two cycles as defined in Definition \ref{def:graph}. The typical matchings of blue half edges is such that in each long red cycle of length greater than $2$, there exists a single long blue cycle connecting one half edge in each niche. The other half edges inside a niche are matched according to a perfect matching. Each long red cycle of length $2q$ gives a contribution $\theta_2 (P_k)^{q}$.
 If the red cycle has length 2 (it is called simple in this case) then the blue edges are matched pairwise arbitrarily. Each simple cycle yields a contribution $\theta_1 (P_k).$ This is illustrated in Figure \ref{fig:matchadmiss}
 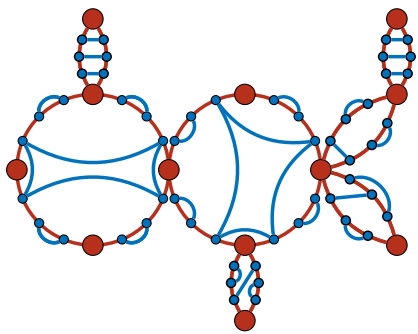
\begin{figure}[!ht]
 	\centering
 		\begin{tikzpicture}
		\draw[-,BrickRed,line width = .12em] (0,0) circle (1cm);
		\draw[-,BrickRed,line width = .12em] (2,0) circle (1cm);
		\node[fill=BrickRed, inner sep = 0pt, minimum size=.27cm] (2) at (1,0) {};
		\node[fill=BrickRed, inner sep = 0pt, minimum size=.27cm] at (-1,0) {};
		\node[fill=BrickRed, inner sep = 0pt, minimum size=.27cm] (1) at (0,1) {};
		\node[fill=BrickRed, inner sep = 0pt, minimum size=.27cm] at (0,-1) {};
		\node[fill=BrickRed, inner sep = 0pt, minimum size=.27cm] (6) at (0,2) {};
		\node[fill=BrickRed, inner sep = 0pt, minimum size=.27cm] at (2,1) {};
		\node[fill=BrickRed, inner sep = 0pt, minimum size=.27cm] (5) at (2,-1) {};
		\node[fill=BrickRed, inner sep = 0pt, minimum size=.27cm] (3) at (3,0) {};
		\node[fill=BrickRed, inner sep = 0pt, minimum size=.27cm] (9) at (2,-2) {};
		\node[fill=BrickRed, inner sep = 0pt, minimum size=.27cm] (4) at (4,1) {};
		\node[fill=BrickRed, inner sep = 0pt, minimum size=.27cm] (8) at (4,-1) {};
		\node[fill=BrickRed, inner sep = 0pt, minimum size=.27cm] (7) at (4,2) {};
		
		\draw[-,BrickRed, line width=.12em] (1) edge[bend left] (6); 
		\draw[-,BrickRed, line width=.12em] (1) edge[bend right] (6); 
		
		\draw[-,BrickRed, line width=.12em] (5) edge[bend left] (9);
		\draw[-,BrickRed, line width=.12em] (5) edge[bend right] (9); 
		
		\draw[-,BrickRed, line width=.12em] (3) edge[bend left] (4); 
		\draw[-,BrickRed, line width=.12em] (3) edge[bend right] (4); 
		
		\draw[-,BrickRed, line width=.12em] (3) edge[bend left] (8); 
		\draw[-,BrickRed, line width=.12em] (3) edge[bend right] (8); 
		
		\draw[-,BrickRed, line width=.12em] (4) edge[bend left] (7); 
		\draw[-,BrickRed, line width=.12em] (4) edge[bend right] (7); 
		\draw[-, BrickRed, line width=.12em] (1) edge[bend right]  node[draw=black,line width=.05em,pos=.20, fill=RoyalBlue, inner sep = 0pt, minimum size = .12cm] (161) {} node[draw=black,line width=.05em,midway, fill=RoyalBlue, inner sep = 0pt, minimum size = .12cm] (162) {} node[draw=black,line width=.05em,pos=.80, fill=RoyalBlue, inner sep = 0pt, minimum size = .12cm] (163) {} (6) ;
		\draw[-, BrickRed, line width=.12em] (1) edge[bend left]  node[draw=black,line width=.05em,pos=.20, fill=RoyalBlue, inner sep = 0pt, minimum size = .12cm] (164) {} node[draw=black,line width=.05em,midway, fill=RoyalBlue, inner sep = 0pt, minimum size = .12cm] (165) {} node[draw=black,line width=.05em,pos=.80, fill=RoyalBlue, inner sep = 0pt, minimum size = .12cm] (166) {} (6) ;
		
		\draw[-, BrickRed, line width=.12em] (5) edge[bend right]  node[draw=black,line width=.05em,pos=.20, fill=RoyalBlue, inner sep = 0pt, minimum size = .12cm] (591) {} node[draw=black,line width=.05em,midway, fill=RoyalBlue, inner sep = 0pt, minimum size = .12cm] (592) {} node[draw=black,line width=.05em,pos=.80, fill=RoyalBlue, inner sep = 0pt, minimum size = .12cm] (593) {} (9) ;
		\draw[-, BrickRed, line width=.12em] (5) edge[bend left]  node[draw=black,line width=.05em,pos=.20, fill=RoyalBlue, inner sep = 0pt, minimum size = .12cm] (594) {} node[draw=black,line width=.05em,midway, fill=RoyalBlue, inner sep = 0pt, minimum size = .12cm] (595) {} node[draw=black,line width=.05em,pos=.80, fill=RoyalBlue, inner sep = 0pt, minimum size = .12cm] (596) {} (9) ;
		
		\draw[-, BrickRed, line width=.12em] (3) edge[bend right]  node[draw=black,line width=.05em,pos=.20, fill=RoyalBlue, inner sep = 0pt, minimum size = .12cm] (341) {} node[draw=black,line width=.05em,midway, fill=RoyalBlue, inner sep = 0pt, minimum size = .12cm] (342) {} node[draw=black,line width=.05em,pos=.80, fill=RoyalBlue, inner sep = 0pt, minimum size = .12cm] (343) {} (4) ;
		\draw[-, BrickRed, line width=.12em] (3) edge[bend left]  node[draw=black,line width=.05em,pos=.20, fill=RoyalBlue, inner sep = 0pt, minimum size = .12cm] (344) {} node[draw=black,line width=.05em,midway, fill=RoyalBlue, inner sep = 0pt, minimum size = .12cm] (345) {} node[draw=black,line width=.05em,pos=.80, fill=RoyalBlue, inner sep = 0pt, minimum size = .12cm] (346) {} (4) ;
		
		\draw[-, BrickRed, line width=.12em] (3) edge[bend right]  node[draw=black,line width=.05em,pos=.20, fill=RoyalBlue, inner sep = 0pt, minimum size = .12cm] (381) {} node[draw=black,line width=.05em,midway, fill=RoyalBlue, inner sep = 0pt, minimum size = .12cm] (382) {} node[draw=black,line width=.05em,pos=.80, fill=RoyalBlue, inner sep = 0pt, minimum size = .12cm] (383) {} (8) ;
		\draw[-, BrickRed, line width=.12em] (3) edge[bend left]  node[draw=black,line width=.05em,pos=.20, fill=RoyalBlue, inner sep = 0pt, minimum size = .12cm] (384) {} node[draw=black,line width=.05em,midway, fill=RoyalBlue, inner sep = 0pt, minimum size = .12cm] (385) {} node[draw=black,line width=.05em,pos=.80, fill=RoyalBlue, inner sep = 0pt, minimum size = .12cm] (386) {} (8) ;

		\draw[-, BrickRed, line width=.12em] (4) edge[bend right]  node[draw=black,line width=.05em,pos=.20, fill=RoyalBlue, inner sep = 0pt, minimum size = .12cm] (471) {} node[draw=black,line width=.05em,midway, fill=RoyalBlue, inner sep = 0pt, minimum size = .12cm] (472) {} node[draw=black,line width=.05em,pos=.80, fill=RoyalBlue, inner sep = 0pt, minimum size = .12cm] (473) {} (7) ;
		\draw[-, BrickRed, line width=.12em] (4) edge[bend left]  node[draw=black,line width=.05em,pos=.20, fill=RoyalBlue, inner sep = 0pt, minimum size = .12cm] (474) {} node[draw=black,line width=.05em,midway, fill=RoyalBlue, inner sep = 0pt, minimum size = .12cm] (475) {} node[draw=black,line width=.05em,pos=.80, fill=RoyalBlue, inner sep = 0pt, minimum size = .12cm] (476) {} (7) ;
		
		\foreach \a in {1,2,...,15}{
		\ifthenelse{\a=4 \OR \a=8 \OR \a=12}{}{\node[fill=RoyalBlue, inner sep = 0pt, minimum size=.12cm] (\a) at (360/4+\a*360/16: 1cm) {};}
		}
		
		\begin{scope}[xshift=2cm]
		\foreach \a in {1,2,...,15}{
		\ifthenelse{\a=4 \OR \a=8 \OR \a=12}{}{\node[fill=RoyalBlue, inner sep = 0pt, minimum size=.12cm] (\a+1) at (360/4+\a*360/16: 1cm) {};}
		}
		\end{scope}
	
		\draw[-, RoyalBlue, line width=.12em] (161) edge  (164);
		\draw[-, RoyalBlue, line width=.12em] (162) edge  (165);
		\draw[-, RoyalBlue, line width=.12em] (163) edge  (166);
		
		\draw[-, RoyalBlue, line width=.12em] (591) edge[bend left=60]  (592);
		\draw[-, RoyalBlue, line width=.12em] (593) edge  (594);
		\draw[-, RoyalBlue, line width=.12em] (595) edge [bend right=60] (596);
		
		\draw[-, RoyalBlue, line width=.12em] (381) edge  (385);
		\draw[-, RoyalBlue, line width=.12em] (384) edge[bend left=60] (386);
		\draw[-, RoyalBlue, line width=.12em] (382) edge[bend right=60]  (383);
		
		\draw[-, RoyalBlue, line width=.12em] (341) edge  (344);
		\draw[-, RoyalBlue, line width=.12em] (342) edge[bend right=60]  (343);
		\draw[-, RoyalBlue, line width=.12em] (345) edge[bend left=60]  (346);
			
		\draw[-, RoyalBlue, line width=.12em] (471) edge  (474);
		\draw[-, RoyalBlue, line width=.12em] (472) edge  (475);
		\draw[-, RoyalBlue, line width=.12em] (473) edge  (476);
		
		\draw[-, RoyalBlue, line width=.12em] (1) edge[bend right=60]  (2);
		\draw[-, RoyalBlue, line width=.12em] (6) edge[bend right=60]  (7);
		\draw[-, RoyalBlue, line width=.12em] (9) edge[bend right=60]  (10);
		\draw[-, RoyalBlue, line width=.12em] (14) edge[bend right=60]  (15);
		
		\draw[-, RoyalBlue, line width=.12em] (3) edge[bend right]  (13);
		\draw[-, RoyalBlue, line width=.12em] (13) edge[bend right]  (11);
		\draw[-, RoyalBlue, line width=.12em] (11) edge[bend right]  (5);
		\draw[-, RoyalBlue, line width=.12em] (5) edge[bend right]  (3);
		
		\draw[-, RoyalBlue, line width=.12em] (1) edge[bend right=60]  (2);
		\draw[-, RoyalBlue, line width=.12em] (6) edge[bend right=60]  (7);
		\draw[-, RoyalBlue, line width=.12em] (9) edge[bend right=60]  (10);
		\draw[-, RoyalBlue, line width=.12em] (14) edge[bend right=60]  (15);

		\draw[-, RoyalBlue, line width=.12em] (3+1) edge[bend right=60]  (2+1);
		\draw[-, RoyalBlue, line width=.12em] (14+1) edge[bend right=60]  (15+1);
		\draw[-, RoyalBlue, line width=.12em] (10+1) edge[bend right=60]  (11+1);
		\draw[-, RoyalBlue, line width=.12em] (6+1) edge[bend right=60]  (5+1);
		
		\draw[-, RoyalBlue, line width=.12em] (1+1) edge[bend right]  (13+1);
		\draw[-, RoyalBlue, line width=.12em] (13+1) edge[bend right]  (9+1);
		\draw[-, RoyalBlue, line width=.12em] (9+1) edge[bend right]  (7+1);
		\draw[-, RoyalBlue, line width=.12em] (7+1) edge[bend right]  (1+1);
	\end{tikzpicture}
	\caption{Example of a matching on the same admissible graph as in Figure \ref{fig:admiss} for $k=3$. The blue matchings are different in the length of the cycles is equal or greater than 2.}
	\label{fig:matchadmiss}
 \end{figure}
\item the contribution of \emph{non admissible }red graphs: there are some additional identifications between the red vertices in such a way that the resulting red graph is not a tree of cycles. The matching of blue edges is made arbitrarily (in such a way that there is still no single blue half edge at the end).
\end{itemize}

\section{High moment asymptotics when \texorpdfstring{$f$}{f} is an odd polynomial\label{sec:3}}
This section is concerned with odd polynomials which are a class of functions such that $\theta_3(f)=0$ and thus, according to Theorem \ref{theo:edge}, where the largest eigenvalue sticks to the bulk.  In particular, we consider a polynomial of degree which can grow with the dimension with the matrix and of Taylor-type:
\[
P_k = \sum_{i=0, i \text{ odd}}^k \frac{a_i}{i!}x^i.
\] 
Here the coefficients $a_i$ are such that \eqref{eq:assum2f} holds true. 

\begin{proposition} \label{prop:highmoment}

Let $0<\alpha_1<\alpha_2$ and $q=q(n_1)$ be a sequence such that $q(n_1)\leqslant (\log n_1)^{1+\alpha_1}$. Assume that $$k\leqslant k_0:=\frac{1}{1+\alpha_2}\frac{\log n_1}{\log\log n_1},$$ then
\[
\mathds{E}\left[
	\mathrm{Tr}\, M_k^{2q}
\right]
=
n_1\mathfrak{m}_{2q}(P_k)(1+o(1)).
\]
Assume that $k\geqslant k_0$, then there exists a constant $C>0$ such that 
\[
\mathds{E}\left[
	\mathrm{Tr}\, M_k^{2q}
\right]
\leqslant
n_1\mathfrak{m}_{2q}(P_{k_0})
\left(
1+o(1)
\right)
\]
where $\mathfrak{m}_q$ is defined by (\ref{eq:defmoment}) and $\mathfrak{m}_{2q}(P_{k_0})$ corresponds to the case where $f=P_{k_0}$.
%denotes the moment where the activation function is given by $P_{k_0}$.
\end{proposition}

The proof of Proposition \ref{prop:highmoment}  actually follows from estimated already obtained in \cite{BePe}.

\begin{proof} We here show Proposition \ref{prop:highmoment}. To that aim, we refer to \cite{BePe}. 
We first assume that $P_k$ is a monomial of odd degree $k$ since this will be the one that contributes to the moments: the extension to an arbitrary odd polynomial then follows by linearity.

We know that for $q$ up to order $(\log n_1)^{1+\alpha_1}$ and $k<k_0$ we have that 
\[
\overline{m}_{2q}^{(P_k)}:=\frac{1}{n_1}\E \Tr M_k^{2q}=\mathfrak{m}_{2q}(P_k)(1+o(1))
\]
which gives the first result of the proposition. This follows from Section 3 in \cite{BePe} and in particular the bounds (3.12), (3.13), (3.14) and (3.16) which hold true  %our convergence of moments is true up to $q$ and $k$ such that
provided  $q^k\ll n_0$ -which is satisfied for our choice of $q$ and $k<k_0$. This can readily be extended to an arbitrary odd polynomial of degree $k<k_0.$ 

However, in order to obtain the appropriate polynomial approximation, we need the degree $k$ to be larger than $\frac{\log n_1}{\log \log n_1}.$ This is not a problem because of our choice of polynomial. Indeed, for such high degrees, the $k!$ normalization makes the contribution of very high degrees negligible. 
%This can be seen, for instance, by considering the parameters $\theta_1$ and $\theta_2$ where the polynomial appears in the moment and in the errors. 
Let $k>k_0$ be chosen, then we can write, if we normalize so that the variances are equal to $1$,
\[
\theta_1(P_k)=\mathds{E}\left[ P_k(\mathcal{N})^2\right]
=
\mathds{E}\left[
	\left(
		P_k(\mathcal{N})-P_{k_0}(\mathcal{N})
	\right)^2
\right]
+
\mathds{E}\left[
	P_{k_0}(\mathcal{N})^2
\right]
+
2\mathds{E}\left[
	P_{k_0}(\mathcal{N})(P_{k}(\mathcal{N})-P_{k_0}(\mathcal{N})
\right]
\]
where $\mathcal{N}$ is a standard Gaussian random variable. By the Cauchy--Schwarz inequality, we now simply need to bound the first term, 
\[
\mathds{E}\left[
	\left(
		P_k(\mathcal{N})-P_{k_0}(\mathcal{N})
	\right)^2
\right]
=
\sum_{\substack{i,j=k_0+1\\i+j\text{ even}}}^k
a_ia_j\frac{(i+j)!!}{i! j!}
\]
which goes to zero exponentially fast by Stirling's formula. Similarly one can show the same holds true for $\theta_2(P_{k})$. 
Thus one has for any $D>0$,
\begin{eqnarray}
&\theta_1(P_k)&=\theta_1(P_{k_0})+\mathcal{O}\left(N^{-D}\right);\cr
&\theta_2(P_k)&=\theta_2(P_{k_0})+\mathcal{O}\left(N^{-D}\right). \end{eqnarray}

Note that, when computing the leading order of the moment $\E \Tr \left (\frac{YY^\top}{m} \right)^q$, this is the only part where the polynomial $P_k$ intervenes, since the admissible graphs do not depend on the activation function. However, as we choose $k$ large, actually large enough so that $q^k \gg n_0$, we need to check that the errors do not explode and actually vanish for $f_{k_0}=\sum_{i>k_0}^k\frac{a_i}{i!}x^i$. One may note that these ``errors'' may contribute more than the corresponding admissible graphs but one needs to show that their contribution is still negligible. One can check from the previous analysis in \cite{BePe} that the largest error comes from (3.12), for such a polynomial $P_{k_0}$.   Actually (3.16) is here replaced with 
\[ n_0\sum_{p=2}^q \left (\frac{Cp^k}{n_0}\right)^p \left (\frac{ \sqrt{(2k)!!}}{k!}\right)^p,\]
when considering the contribution of non admissible graphs. 
%, we can see that 
We thus need to bound the two quantities for $k_i,k_j>k_0$
\[
\frac{k_i(k_i-1)!!}{k_i!}\frac{q^{k_i}}{n_0}\quad\text{and}\quad \frac{(k_i+k_j)!!}{k_i!k_j!}\frac{q^{k_i+k_j}}{n_0}.
\]
We bound the first one but the second one can be bounded in the same way. Note that these bounds come from the two different behaviors in the case of a cycle of length 2 and larger cycles. Now  using Stirling's formula we can see that
\[
\frac{k_i(k_i-1)!!}{k_i!}\frac{q^{k_i}}{n_0}
=
\mathcal{O}\left(
\frac{\sqrt{k_i}}{n_0}
\left(
\sqrt{e}
\frac{q}{\sqrt{k_i}}
\right)^{k_i}
\right).
\]
This bound is decreasing in $k_i$ so that we need to check its order for $k_i=k_0=\frac{1}{1+\alpha_2}\frac{\log n_1}{\log\log n_1}.$ And we obtain the following bound,
\[
\frac{k_i(k_i-1)!!}{k_i!}\frac{q^{k_i}}{n_0}
=
\mathcal{O}\left(
\frac{\psi(n_1)}{n_1^{-\frac{1+2(\alpha_2-\alpha_1)}{2(1+\alpha_2)}}}
\right),
\]
with the function $\psi$ given by
\[
\psi(n_1)
=
\sqrt{
	\frac{1}{1+\alpha_2}\frac{\log n_1}{\log\log n_1}
}
\left(
	\frac{e}{1+\alpha_2}\log\log n_1
\right)^{\frac{1}{2(1+\alpha_2)}\frac{\log n_1}{\log\log n_1}}
=
\mathcal{O}(n_1^\varepsilon),
\]
for any $\varepsilon>0$. Thus, recalling that $\alpha_2>\alpha_1$ we have that, by taking $\varepsilon$ small enough, %for any $\varepsilon>0$,
\[
\frac{k_i(k_i-1)!!}{k_i!}\frac{q^{k_i}}{n_0}
=
\mathcal{O}\left(
\frac{n_1^\varepsilon}{n_1^{-\frac{1+2(\alpha_2-\alpha_1)}{2(1+\alpha_2)}}}
\right)
=
o(1).
\]
\end{proof}

\section{Behavior of the largest eigenvalue when $f$ is an even polynomial} \label{sec:4}
This section is devoted to the proof of Theorem \ref{theo:sep_easy} when $f=P_{2k}$ is a polynomial of degree $2k$: 
%In particular we have $\theta_2(f)=0$ and we know from \cite{BePe} that the asymptotic empirical eigenvalue distribution of $M$ is given by the Marchenko--Pastur distribution. Thus, when considering the spectral moments $\Tr M^q$ for any finite $q$, only admissible graphs with cycle of length 2 contributes to the moments which coincides with Narayana numbers. Even monomials contribute differently when consider high moments of order growing with $n_1$ as matchings which do not contribute for finite $q$ can become contributing for $q$ growing faster than $\log n_1$.

 $$P_{2k}(x)=x^{2k}-(2k)!!.$$  In particular we have $\theta_2(f)=0$ 
 The aim of this section is to show that the largest eigenvalue may separate from the rest of the spectrum for such an activation function. \\
 
 We first start with some reminders before entering the proof. \\
 Let us consider the moment of order $q$ of the e.e.d. of the matrix $M_{2k}$, namely $\mathbb{E}\text{Tr}M_{2k}^q$. It was already shown in \cite{BePe} that, when $q$ is independent of $n_0$, the main contribution comes from the cactus graphs whose $q$ fundamental cycles are simple i.e. have length 2 only. Blue edges are then matched tipically inside these simple cycles, independently of the two niches. 
 The contribution of other admissible graphs is then negligible in the large $n_0$-limit. 
 The above statements are equivalent to the fact that the limiting empirical eigenvalue distribution is the Marchenko--Pastur distribution with support 
 $$( {\bf{u_-}, \bf{u_+}})=\theta_1(P_{2k}) (u_-, u_+)\quad\text{where}\quad  u_{\pm}=\left (1\pm \sqrt \gamma \right )^2\quad \text{with}\quad \gamma=\frac{\phi}{\psi}.$$
 The number of double trees of simple cycles of total length $2q$ with $l$ distinct $j$-indices is given by the Narayana number
$$\frac{1}{q}\binom{q}{l}\binom{q}{l-1}, \:l \geq 1.$$
Thus, when $q$ is a fixed integer (independent of $n_0$), one has that
$$\frac{1}{n_1}\mathbb{E}\text{Tr}M_{2k}^q= \sum_{l\geq 1}\frac{1}{q}\binom{q}{l}\binom{q}{l-1}\frac{1}{\gamma^{q-l}}\theta_1(P_{2k})^q(1+o(1)),$$
which means that the trees of simple cycles are the only contributing graphs, as shown in \cite{BePe}.

The core of the proof of Theorem \ref{theo:sep_easy} in the case where $f$ is an even polynomial is the following proposition.
Define \begin{equation}\kappa_x=\kurt(X)-1,\quad \kappa_w=\kurt(W)-1,\quad\text{and}\quad \kappa=\max(\kurt(W)-1, \kurt (X)-1).\label{kappa}\end{equation}
\begin{proposition}\label{prop:even} Let $\alpha>0$, $k$ be a fixed integer independent of $n_0$ and $q\leqslant (\log n_1)^{1+\alpha_1}$. Assume that $P_{2k}(x)=x^{2k}-(2k)!!$ then 
\[
\mathbb{E}\Tr M_{2k}^q= \mathbb{E}\Tr\left (\frac{1}{m}\left(\tilde{X}+\sqrt{\frac{\theta_3\kappa}{n_0}} J\right)\left(\tilde{X}+\sqrt{\frac{\theta_3\kappa}{n_0}} J)\right)^\top\right )^q(1+o(1))
\]
where $\tilde{X}$ is a $n_1 \times m$ matrix with i.i.d. $\mathcal{N}(0,1)$ entries.
\end{proposition}

\begin{proof}[Proof of Proposition \ref{prop:even}:]

We recall that, for simplicity, we set $\sigma_w=\sigma_x=1$ and that
\begin{multline}\label{eq:entryy}
Y_{ij}=\left(\frac{WX}{\sqrt{n_0}}\right)^{2k}_{ij}-(2k)!!=\frac{1}{n_0^{k}}\left(
	\sum_{\ell=1}^{n_0}W_{ik}X_{kj}
\right)^{2k}-2k!!
\\
=
\frac{1}{n_0^{k}}\sum_{\ell_1,\dots\ell_{2k}=1}^{n_0}\prod_{p=1}^kW_{i\ell_p}X_{\ell_p j}-\#\left\{\text{perfect matchings of }\unn{1}{2k}\right\}.
\end{multline} 

We also recall that the $\ell$-indices have to be matched inside niches first: this is due to the fact that the $W$ and $X$ entries are centered. We denote by $\mathrm{PM_{2k}}$ the perfect matchings of $\unn{1}{2k}$ and $\mathrm{NPM}_{2k}$ those matchings which are not perfect matchings. Then the above can be rewritten as 
\begin{equation} (\ref{eq:entryy})=
\frac{1}{n_0^{k}}\sum_{\mathrm{PM}_{2k}}\sum_{\ell_1,\dots\ell_{k}=1^* }^{n_0}\left(\prod_{p=1}^kW_{i\ell_p}^2X_{\ell_p j}^2-1\right) + \frac{1}{n_0^{k}}\sum_{\mathrm{NPM}_{2k}}\sum_{\ell_1,\dots\ell_{2k}=1^* }^{n_0}\prod_{p=1}^kW_{i\ell_p}X_{\ell_p j}.
\end{equation}

In the above formula, the first starred sum bears on pairwise distinct $\ell$-indices. The second starred sum bears on a certain number of pairwise distinct $\ell$-indices which depends on the matching of the $2k$ integers defined previously.

Consider now the expected value of 
\begin{equation}\label{eq:moment0}
\frac{1}{n_1}\mathds{E}\left[
	\Tr M_{2k}^q
\right]
=
\frac{1}{n_1m^q}\mathds{E}\left[
	\Tr \left(YY^\top\right)^q
\right]
=
\frac{1}{n_1m^q}\mathds{E}\sum_{i_1,\dots,i_q=1}^{n_1}\sum_{j_1,\dots,j_q=1}^mY_{i_1j_1}Y_{i_2j_1}Y_{i_2j_2}Y_{i_3j_2}\dots Y_{i_qj_q}Y_{i_1j_q}.
\end{equation}

Assume that there exists at least one niche where the $\ell$-indices are matched according to a perfect matching and that this niche is not connected to another by a bridge. Then the corresponding contribution to the expected value vanishes. As a result, we can now restrict to matchings which are either different from a perfect matching or such that matchings inside niches are connected to another through a bridge. 

\vspace{\baselineskip}
\begin{minipage}{.45\linewidth}
	\begin{center}
	\begin{tikzpicture}
		\node[fill=BrickRed, inner sep = 0pt, minimum size=.27cm] (A) at (-1.5,0) {};
		\node[fill=BrickRed, inner sep=0pt, minimum size=.27cm] (B) at (0,1.5) {};
		\draw[-,BrickRed, line width = .12em] (A) arc[start angle=180, end angle = 90, radius=1.5cm];
		\foreach \a in {1,2,3,4}{
		\node[fill=RoyalBlue, inner sep = 0pt, minimum size=.12cm] (\a+1) at (360/4+\a*360/20: 1.5cm) {};		
		}

		\draw[-, RoyalBlue, line width=.12em] (1+1) edge[bend left=60] (3+1);
		\draw[-, RoyalBlue, line width=.12em] (2+1) edge[bend right=60] (4+1);
	\end{tikzpicture}
	\end{center}
\end{minipage}
\begin{minipage}{.45\linewidth}
	Such a perfect matching in any niche is forbidden by the centering and thus does not contribute to the moment.
\end{minipage}
\vspace{\baselineskip}

As perfect matchings are forbidden, and in order to maximize the number of pairwise distinct indices, there are two possible cases. Either some blue vertices are matched with vertices from other niches and there
 exist more than one blue cycle linking the niches.
Or there are additional matchings between adjacent niches, that we call \emph{bridges}. Note that this holds true whether the graph is admissible or not.

We first consider the contribution of admissible graphs. 
Because $q$ grows to infinity but not too fast, some typical graphs are the same as in the case where $q$ is independent of $n_0$, i.e. trees of simple cycles. 
Consider now typical graphs $G$ with $q$ fundamental cycles of length $2$ and where the $\ell$-indices are matched according to a perfect matching that crosses the niches (i.e. with extra niches matchings). Their total contribution to $\mathbb{E}\text{Tr}M_{2k}^q$ is in the order of $n_1\theta_1^q u_+^{q}.$ This is in agreement with the fact that the limiting e.e.d. is the Marchenko--Pastur distribution.

\vspace{\baselineskip}
\begin{minipage}{.45\linewidth}
	\centering
	\begin{tikzpicture}
		\node[fill=BrickRed, inner sep=0pt, minimum size = .27cm] (1) at (0,0) {};
		\node[fill=BrickRed, inner sep=0pt, minimum size = .27cm] (2) at (2.5,0) {};
		\draw[-, BrickRed, line width=.12em] (1) edge[bend left]  node[draw=black,line width=.05em,pos=.20, fill=RoyalBlue, inner sep = 0pt, minimum size = .12cm] (121) {} node[draw=black,line width=.05em,pos=.4, fill=RoyalBlue, inner sep = 0pt, minimum size = .12cm] (122) {} node[draw=black,line width=.05em,pos=.60, fill=RoyalBlue, inner sep = 0pt, minimum size = .12cm] (123) {} node[draw=black,line width=.05em,pos=.80, fill=RoyalBlue, inner sep = 0pt, minimum size = .12cm] (124) {} (2) ;
		
		\draw[-, BrickRed, line width=.12em] (1) edge[bend right]  node[draw=black,line width=.05em,pos=.20, fill=RoyalBlue, inner sep = 0pt, minimum size = .12cm] (125) {} node[draw=black,line width=.05em,pos=.4, fill=RoyalBlue, inner sep = 0pt, minimum size = .12cm] (126) {} node[draw=black,line width=.05em,pos=.60, fill=RoyalBlue, inner sep = 0pt, minimum size = .12cm] (127) {} node[draw=black,line width=.05em,pos=.80, fill=RoyalBlue, inner sep = 0pt, minimum size = .12cm] (128) {} (2) ;
		
		\draw[-, RoyalBlue, line width=.12em] (121) edge (128);
		\draw[-, RoyalBlue, line width=.12em] (122) edge[bend left=60] (124);
		\draw[-, RoyalBlue, line width=.12em] (123) edge (125);
		\draw[-, RoyalBlue, line width=.12em] (126) edge[bend right=80] (127);
	\end{tikzpicture}
\end{minipage}
\begin{minipage}{.45\linewidth}
	In a cycle of length 2, we perform a perfect matching between the $4k$ blue vertices while preventing a perfect matching in each niche. At least one blue matching (and thus two) must cross niches.
\end{minipage}

\vspace{\baselineskip}
However, because $q$ grows to infinity there may exist other typical graphs. We are going to show that admissible graphs with a single cycle of length $2q'>4$ with attached trees of simple cycles may lead to a non negligible contribution.

Firstly we consider the contribution of graphs where, for some long cycle, (at least) two long blue cycles match the $\ell$-indices. 
Consider a long cycle, first of length $2q'\geqslant 4$.  Such a cycle has been obtained by choosing $q'$ adjacent simple cycles 
from a tree of simple cycles. One simply opens the cycles. Because there are (at least) two long blue cycles,  one can check that the total number of distinct vertices is then in the order of $n_0^{-q'+1}$ that corresponding to the $q'$ simple cycles. 

\vspace{\baselineskip}
\begin{minipage}{.45\linewidth}
	\centering
		\begin{tikzpicture}
			\draw[-,BrickRed,line width = .12em] (0,0) circle (1cm);
			\node[fill=BrickRed,inner sep=0pt, minimum size=.27cm] (2) at (1,0) {};
			\node[fill=BrickRed,inner sep=0pt, minimum size=.27cm] at (-1,0) {};
			\node[fill=BrickRed,inner sep=0pt, minimum size=.27cm] (1) at (0,1) {};
			\node[fill=BrickRed,inner sep=0pt, minimum size=.27cm] at (0,-1) {};
			\foreach \a in {1,2,...,19}{
				\ifthenelse{\a=5 \OR \a=10 \OR \a=15}{}{\node[fill=RoyalBlue, inner sep = 0pt, minimum size=.12cm] (\a+1) at (360/4+\a*360/20: 1cm) {};}
			}
			
			\draw[-, RoyalBlue, line width=.12em] (1+1) edge[bend left=20]  (7+1);
			\draw[-, RoyalBlue, line width=.12em] (7+1) edge[bend left=20]  (13+1);
			\draw[-, RoyalBlue, line width=.12em] (13+1) edge[bend left=20]  (16+1);
			\draw[-, RoyalBlue, line width=.12em] (16+1) edge[bend left=20]  (1+1);
			
			\draw[-, RoyalBlue, line width=.12em] (2+1) edge[bend left=20]  (6+1);
			\draw[-, RoyalBlue, line width=.12em] (6+1) edge[bend left=20]  (11+1);
			\draw[-, RoyalBlue, line width=.12em] (11+1) edge[bend left=20]  (17+1);
			\draw[-, RoyalBlue, line width=.12em] (17+1) edge[bend left=20]  (2+1);
			
			\draw[-, RoyalBlue, line width=.12em] (3+1) edge[bend left=80]  (4+1);
			\draw[-, RoyalBlue, line width=.12em] (8+1) edge[bend left=80]  (9+1);
			\draw[-, RoyalBlue, line width=.12em] (12+1) edge[bend right=80]  (14+1);
			\draw[-, RoyalBlue, line width=.12em] (18+1) edge[bend left=80]  (19+1);
	\end{tikzpicture}
\end{minipage}
\begin{minipage}{.45\linewidth}
	At least two blue cycles link each niche together. In particular, not a single niche consists in just a perfect matching. This contributes $(4\theta_3(f))^{q}$ with $2q$ the length of the cycle.
\end{minipage}

\vspace{\baselineskip}
One can then write the contribution of such matched graphs to $\mathbb{E}\text{Tr}M_{2k}^q$ as 

\begin{equation}\label{err}
\sum_{I_i,I_j=0}^q\sum_{b=0}^{I_i+I_j}{\mathcal{A}}(q,I_i,I_j,b)(4\theta_3)^{q-b}\theta_1^{b}\psi^{I_i+1-q}\phi^{I_j}n_0^{b+c-(q-1)},
\end{equation}
where $c=I_i+I_j+1-b$ denotes the number of long cycles. 
It is not difficult to check that the above sum does not exceed $\sum_l \binom{q}{l}n_0^{-(l-1)}$ times the contribution of trees of simple cycles. 
As a consequence this is at most in the order of $q/n_0$ times the contribution of trees of simple cycles. 
Combining the whole, one can check that the contribution of admissible cycles with one long cycle and two long blue cycles is negligible in the large $n_0$ limit. 
%Again we can do the same analysis for any other matchings inside niches which is not a perfect matching in any of the niches. 

%Now, so that a graph contributes in the large $n_0$-limit, one needs the power of $\hat{\theta_2}$ to be in the order of $q$. This implies that the number of $l$-indices decreases much and is negligible. Short cycles of length $4$ yield a negligible contribution due to the cost of order $n_0^{-1}$ induced by the presence of a single long cycle. Combining the whole, one can check that the contribution of admissible cycles with one long cycle and two long blue cycle is negligible in the large $n_0$-limit. 

\vspace{\baselineskip}
Now we need to consider the contribution from admissible graphs which have bridges between adjacent niches. A bridge is an identification of two $\ell$-indices from two adjacent niches. Note that the identification can be made around an $i$ or a $j$ index. Starting from a perfect matching inside niches,  such an identification gives rise to the occurence of a fourth moment either of a $W$ entry (identification around an $i$-index) or an $X$ entry (similarly around a $j$-index).
For other matchings, note that this may give rise to higher order moments.

We now consider the contribution of the following graphs which we prove are typical too: consider an admissible graph with $b$ cycles of length $2$.
For such graphs, we first consider the following matchings illustrated in Figure \ref{fig:matchadmisseven}:
\begin{itemize}
\item In any fundamental cycle of lenth $2q_1\geq 4$ one performs a perfect matching inside each niche and adds  $q_1$ bridges between pairs of adjacent niches so that no niche is disconnected from its neighbors. Thus for a cycle of length $2q_1$ one has to choose $1$ matching inside each niche and then decide once around which kind of vertex the bridges are built. In the whole there are $k^{2q_1}(2k)!!^{2q_1}=(k(2k-1)(wk-2)!!)^{2q_1}=\theta_3(P_{2k})$ such choices and matchings. By construction, the corresponding moment is for each bridge around an $i$-index
$$\mathbb{E} (W_{i\ell}^2X_{\ell j}^2-1)(W_{i\ell}^2X_{\ell j'}^2-1)=\kurt(W)-1=\kappa_w
$$
or $$\mathbb{E} (W_{i\ell}^2X_{\ell j}^2-1)(W_{i'\ell}^2X_{\ell j}^2-1)=\kurt(X)-1=\kappa_x,
$$
 around a $j$-index.

\item In any cycle of length $2$ one performs a perfect matching in such a way that there is at least one matching from one niche to the other. 
\end{itemize}

\begin{figure}
 	\centering
 		\begin{tikzpicture}
		\draw[-,BrickRed,line width = .12em] (0,0) circle (1cm);
		\draw[-,BrickRed,line width = .12em] (2,0) circle (1cm);
		\node[fill=BrickRed, inner sep = 0pt, minimum size=.27cm] (2) at (1,0) {};
		\node[fill=BrickRed, inner sep = 0pt, minimum size=.27cm] at (-1,0) {};
		\node[fill=BrickRed, inner sep = 0pt, minimum size=.27cm] (1) at (0,1) {};
		\node[fill=BrickRed, inner sep = 0pt, minimum size=.27cm] at (0,-1) {};
		\node[fill=BrickRed, inner sep = 0pt, minimum size=.27cm] (6) at (0,2) {};
		\node[fill=BrickRed, inner sep = 0pt, minimum size=.27cm] at (2,1) {};
		\node[fill=BrickRed, inner sep = 0pt, minimum size=.27cm] (5) at (2,-1) {};
		\node[fill=BrickRed, inner sep = 0pt, minimum size=.27cm] (3) at (3,0) {};
		\node[fill=BrickRed, inner sep = 0pt, minimum size=.27cm] (9) at (2,-2) {};
		\node[fill=BrickRed, inner sep = 0pt, minimum size=.27cm] (4) at (4,1) {};
		\node[fill=BrickRed, inner sep = 0pt, minimum size=.27cm] (8) at (4,-1) {};
		\node[fill=BrickRed, inner sep = 0pt, minimum size=.27cm] (7) at (4,2) {};
		
		\draw[-,BrickRed, line width=.12em] (1) edge[bend left] (6); 
		\draw[-,BrickRed, line width=.12em] (1) edge[bend right] (6); 
		
		\draw[-,BrickRed, line width=.12em] (5) edge[bend left] (9);
		\draw[-,BrickRed, line width=.12em] (5) edge[bend right] (9); 
		
		\draw[-,BrickRed, line width=.12em] (3) edge[bend left] (4); 
		\draw[-,BrickRed, line width=.12em] (3) edge[bend right] (4); 
		
		\draw[-,BrickRed, line width=.12em] (3) edge[bend left] (8); 
		\draw[-,BrickRed, line width=.12em] (3) edge[bend right] (8); 
		
		\draw[-,BrickRed, line width=.12em] (4) edge[bend left] (7); 
		\draw[-,BrickRed, line width=.12em] (4) edge[bend right] (7); 
		\draw[-, BrickRed, line width=.12em] (1) edge[bend right]  node[draw=black,line width=.05em,pos=.20, fill=RoyalBlue, inner sep = 0pt, minimum size = .12cm] (161) {} node[draw=black,line width=.05em,pos=.4, fill=RoyalBlue, inner sep = 0pt, minimum size = .12cm] (162) {} node[draw=black,line width=.05em,pos=.60, fill=RoyalBlue, inner sep = 0pt, minimum size = .12cm] (163) {} node[draw=black,line width=.05em,pos=.80, fill=RoyalBlue, inner sep = 0pt, minimum size = .12cm] (164) {} (6) ;
		\draw[-, BrickRed, line width=.12em] (1) edge[bend left]  node[draw=black,line width=.05em,pos=.20, fill=RoyalBlue, inner sep = 0pt, minimum size = .12cm] (165) {} node[draw=black,line width=.05em,pos=.4, fill=RoyalBlue, inner sep = 0pt, minimum size = .12cm] (166) {} node[draw=black,line width=.05em,pos=.60, fill=RoyalBlue, inner sep = 0pt, minimum size = .12cm] (167) {} node[draw=black,line width=.05em,pos=.80, fill=RoyalBlue, inner sep = 0pt, minimum size = .12cm] (168) {} (6) ;
		
		\draw[-, BrickRed, line width=.12em] (5) edge[bend right]  node[draw=black,line width=.05em,pos=.20, fill=RoyalBlue, inner sep = 0pt, minimum size = .12cm] (591) {} node[draw=black,line width=.05em,pos=.4, fill=RoyalBlue, inner sep = 0pt, minimum size = .12cm] (592) {} node[draw=black,line width=.05em,pos=.60, fill=RoyalBlue, inner sep = 0pt, minimum size = .12cm] (593) {} node[draw=black,line width=.05em,pos=.80, fill=RoyalBlue, inner sep = 0pt, minimum size = .12cm] (594) {} (9) ;
		\draw[-, BrickRed, line width=.12em] (5) edge[bend left]  node[draw=black,line width=.05em,pos=.20, fill=RoyalBlue, inner sep = 0pt, minimum size = .12cm] (595) {} node[draw=black,line width=.05em,pos=.4, fill=RoyalBlue, inner sep = 0pt, minimum size = .12cm] (596) {} node[draw=black,line width=.05em,pos=.60, fill=RoyalBlue, inner sep = 0pt, minimum size = .12cm] (597) {} node[draw=black,line width=.05em,pos=.80, fill=RoyalBlue, inner sep = 0pt, minimum size = .12cm] (598) {} (9) ;
		
		\draw[-, BrickRed, line width=.12em] (3) edge[bend right]  node[draw=black,line width=.05em,pos=.20, fill=RoyalBlue, inner sep = 0pt, minimum size = .12cm] (341) {} node[draw=black,line width=.05em,pos=.4, fill=RoyalBlue, inner sep = 0pt, minimum size = .12cm] (342) {} node[draw=black,line width=.05em,pos=.60, fill=RoyalBlue, inner sep = 0pt, minimum size = .12cm] (343) {} node[draw=black,line width=.05em,pos=.80, fill=RoyalBlue, inner sep = 0pt, minimum size = .12cm] (344) {} (4) ;
		\draw[-, BrickRed, line width=.12em] (3) edge[bend left]  node[draw=black,line width=.05em,pos=.20, fill=RoyalBlue, inner sep = 0pt, minimum size = .12cm] (345) {} node[draw=black,line width=.05em,pos=.4, fill=RoyalBlue, inner sep = 0pt, minimum size = .12cm] (346) {} node[draw=black,line width=.05em,pos=.60, fill=RoyalBlue, inner sep = 0pt, minimum size = .12cm] (347) {} node[draw=black,line width=.05em,pos=.80, fill=RoyalBlue, inner sep = 0pt, minimum size = .12cm] (348) {} (4) ;
		
		\draw[-, BrickRed, line width=.12em] (3) edge[bend right]  node[draw=black,line width=.05em,pos=.20, fill=RoyalBlue, inner sep = 0pt, minimum size = .12cm] (381) {} node[draw=black,line width=.05em,pos=.4, fill=RoyalBlue, inner sep = 0pt, minimum size = .12cm] (382) {} node[draw=black,line width=.05em,pos=.60, fill=RoyalBlue, inner sep = 0pt, minimum size = .12cm] (383) {} node[draw=black,line width=.05em,pos=.80, fill=RoyalBlue, inner sep = 0pt, minimum size = .12cm] (384) {} (8) ;
		\draw[-, BrickRed, line width=.12em] (3) edge[bend left]  node[draw=black,line width=.05em,pos=.20, fill=RoyalBlue, inner sep = 0pt, minimum size = .12cm] (385) {} node[draw=black,line width=.05em,pos=.4, fill=RoyalBlue, inner sep = 0pt, minimum size = .12cm] (386) {} node[draw=black,line width=.05em,pos=.60, fill=RoyalBlue, inner sep = 0pt, minimum size = .12cm] (387) {} node[draw=black,line width=.05em,pos=.80, fill=RoyalBlue, inner sep = 0pt, minimum size = .12cm] (388) {} (8) ;

		\draw[-, BrickRed, line width=.12em] (4) edge[bend right]  node[draw=black,line width=.05em,pos=.20, fill=RoyalBlue, inner sep = 0pt, minimum size = .12cm] (471) {} node[draw=black,line width=.05em,pos=.4, fill=RoyalBlue, inner sep = 0pt, minimum size = .12cm] (472) {} node[draw=black,line width=.05em,pos=.60, fill=RoyalBlue, inner sep = 0pt, minimum size = .12cm] (473) {} node[draw=black,line width=.05em,pos=.80, fill=RoyalBlue, inner sep = 0pt, minimum size = .12cm] (474) {} (7) ;
		\draw[-, BrickRed, line width=.12em] (4) edge[bend left]  node[draw=black,line width=.05em,pos=.20, fill=RoyalBlue, inner sep = 0pt, minimum size = .12cm] (475) {} node[draw=black,line width=.05em,pos=.4, fill=RoyalBlue, inner sep = 0pt, minimum size = .12cm] (476) {} node[draw=black,line width=.05em,pos=.60, fill=RoyalBlue, inner sep = 0pt, minimum size = .12cm] (477) {} node[draw=black,line width=.05em,pos=.80, fill=RoyalBlue, inner sep = 0pt, minimum size = .12cm] (478) {} (7) ;
		
		\foreach \a in {1,2,...,19}{
		\ifthenelse{\a=5 \OR \a=10 \OR \a=15}{}{\node[fill=RoyalBlue, inner sep = 0pt, minimum size=.12cm] (\a+2) at (360/4+\a*360/20: 1cm) {};}
		}
		
		\begin{scope}[xshift=2cm]
		\foreach \a in {1,2,...,19}{
		\ifthenelse{\a=5 \OR \a=10 \OR \a=15}{}{\node[fill=RoyalBlue, inner sep = 0pt, minimum size=.12cm] (\a+1) at (360/4+\a*360/20: 1cm) {};}
		}
		\end{scope}
	
		\draw[-, RoyalBlue, line width=.12em] (161) edge  (165);
		\draw[-, RoyalBlue, line width=.12em] (162) edge  (166);
		\draw[-, RoyalBlue, line width=.12em] (163) edge  (168);
		\draw[-, RoyalBlue, line width=.12em] (164) edge  (167);
		
		\draw[-, RoyalBlue, line width=.12em] (591) edge[bend right=80]  (592);
		\draw[-, RoyalBlue, line width=.12em] (593) edge  (595);
		\draw[-, RoyalBlue, line width=.12em] (596) edge [bend left=80] (597);
		\draw[-, RoyalBlue, line width=.12em] (592) edge  (598);

		\draw[-, RoyalBlue, line width=.12em] (381) edge  (385);
		\draw[-, RoyalBlue, line width=.12em] (384) edge (386);
		\draw[-, RoyalBlue, line width=.12em] (382) edge[bend right=80]  (383);
		\draw[-, RoyalBlue, line width=.12em] (387) edge[bend left=80]  (388);

		\draw[-, RoyalBlue, line width=.12em] (341) edge[bend left=60]  (344);
		\draw[-, RoyalBlue, line width=.12em] (342) edge[bend right=60]  (343);
		\draw[-, RoyalBlue, line width=.12em] (345) edge[bend left=80]  (346);
		\draw[-, RoyalBlue, line width=.12em] (347) edge[bend left=80]  (348);

		\draw[-, RoyalBlue, line width=.12em] (471) edge  (477);
		\draw[-, RoyalBlue, line width=.12em] (472) edge  (475);
		\draw[-, RoyalBlue, line width=.12em] (473) edge  (476);
		\draw[-, RoyalBlue, line width=.12em] (474) edge  (478);

		\draw[-, RoyalBlue, line width=.12em] (1+2) edge[bend left=80] node[draw=none, midway, inner sep=0pt, minimum size=.12cm] (a) {} (2+2);
		\draw[-, RoyalBlue, line width=.12em] (3+2) edge[bend left=80] (4+2);
		
		\draw[-, RoyalBlue, line width=.12em] (6+2) edge[bend left=60] node[draw=none, midway, inner sep=0pt, minimum size=.12cm] (b) {} (8+2);
		\draw[-, RoyalBlue, line width=.12em] (7+2) edge[bend right=60]  (9+2);
		\draw[-, RoyalBlue, line width=.12em] (11+2) edge[bend left=80]  (12+2);
		\draw[-, RoyalBlue, line width=.12em] (13+2) edge[bend left=80] node[draw=none, midway, inner sep=0pt, minimum size=.12cm] (c) {}  (14+2);
		
		\draw[-, RoyalBlue, line width=.12em] (16+2) edge[bend left=60] node[draw=none, midway, inner sep=0pt, minimum size=.12cm] (d) {} (19+2);
		\draw[-, RoyalBlue, line width=.12em] (17+2) edge[bend left=60]  (18+2);
		\draw[-, RoyalBlue, line width=.12em] (1+2) edge[bend left=30]  (8+2);
		\draw[-, RoyalBlue, line width=.12em] (2+2) edge[bend left=60]  (6+2);
		
		\draw[-, RoyalBlue, line width=.12em] (11+2) edge[bend left=20]  (19+2);
		\draw[-, RoyalBlue, line width=.12em] (12+2) edge[bend left=60]  (16+2);
		
		\draw[-, RoyalBlue, line width=.12em] (6+1) edge[bend left=80] node[draw=none, midway, inner sep=0pt, minimum size=.12cm] (aa) {} (7+1);
		\draw[-, RoyalBlue, line width=.12em] (8+1) edge[bend left=80] (9+1);
		
		\draw[-, RoyalBlue, line width=.12em] (16+1) edge[bend left=60] node[draw=none, midway, inner sep=0pt, minimum size=.12cm] (bb) {} (18+1);
		\draw[-, RoyalBlue, line width=.12em] (17+1) edge[bend right=60]  (19+1);
		\draw[-, RoyalBlue, line width=.12em] (1+1) edge[bend left=80]  (2+1);
		\draw[-, RoyalBlue, line width=.12em] (3+1) edge[bend left=80] node[draw=none, midway, inner sep=0pt, minimum size=.12cm] (cc) {}  (4+1);
		
		\draw[-, RoyalBlue, line width=.12em] (11+1) edge[bend left=60] node[draw=none, midway, inner sep=0pt, minimum size=.12cm] (dd) {} (14+1);
		\draw[-, RoyalBlue, line width=.12em] (12+1) edge[bend left=60]  (13+1);
		\draw[-, RoyalBlue, line width=.12em] (4+1) edge[bend left=60]  (6+1);
		\draw[-, RoyalBlue, line width=.12em] (3+1) edge[bend left=60]  (7+1);
		\draw[-, RoyalBlue, line width=.12em] (11+1) edge[bend left=60]  (18+1);
		\draw[-, RoyalBlue, line width=.12em] (14+1) edge[bend left=60]  (16+1);

	\end{tikzpicture}
	\caption{Example of a leading matching for $k=4$. In a long cycle, we first perform a perfect matching in each niche and then add some additional identifications between each niche called a \emph{bridge}. In a cycle of length 2, we perform the usual matching. Each cycle of length 2 contributes a factor of $\theta_1(P_k)$ while a cycle of length $q>2$ contributes $n_1^{-1}\psi^{-q}\theta_3(P_k)^q$.}
	\label{fig:matchadmisseven}
\end{figure}
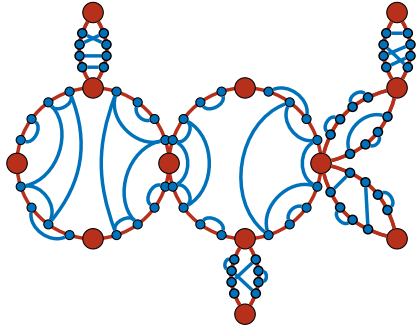
Let $\kappa$ be given by (\ref{kappa}). %=\max(\kurt(W)-1,\kurt(X)-1).$
Combining the whole, the  contribution of any such graphs to $\mathbb{E}\text{Tr}M_{2k}^q$ is then at most

\begin{equation}
\frac{1}{m^q n_0^{2kq}} n_1^{q-I_i}m^{q-I_j}n_0^{(q-b)(2k-1)}n_0^{2kb} \theta_1^{b}\theta_3^{q-b} \kappa^{q-b}.
\end{equation}
Indeed for each long cycle, bridges are built around $i$-indices or $j$-indices but this can differ from one long cycle to another.

We now need to combine the contribution of all long cycles in admissible graphs. 
Let $N=I_i+I_j+1-b$ denote the number of long cycles (we recall that $I_i+I_j+1$ is the total number of cycles in an admissible graph). The vertices where cycles join play a distinct role: a bridge around such a vertex can be made between adjacent niches from two different cycles. This has no impact on the order of the number of pairwise distinct vertices provided these cycles are both long. It is negligible otherwise.
At each vertex where long cycles join, we have to determine the adjacent niches where the bridges are possibly built. There are at most $d_{\max}^2$ such choices where $d_{\max}$ is the maximal degree of such a vertex. Note that $d_{\max}<q$ and that if the graph has only one long cycle, $d_{\max}=1$. 
The total number of admissible graphs here is $\mathcal{A}(q, I_i,I_j,b).$
Thus, one gets that the contribution of admissible graphs to $\mathbb{E}\text{Tr}M_{2k}^q$ is
\begin{eqnarray}&& \mathbf{u}_+^q \theta_1^{q}n_1+\sum_{ I_i, I_j, b<q} \mathcal{A}(q, I_i,I_j,b)\frac{1}{m^q n_0^{2kq}} n_1^{q-I_i}m^{q-I_j}n_0^{2kq -(q-b)} \theta_1^{b}\theta_3^{q-b}d_{\max}^{2(N-1)}
  \kappa^{q-b}\O{1}\cr
&&=n_1\mathbf{u}_+^q \theta_1^q+\sum_{ I_i, I_j, b<q} \mathcal{A}(q, I_i,I_j,b) \psi^{-q+I_i} \phi^{I_j} \theta_1^{b}\theta_3^{q-b} n_0^{b-I_i-I_j}\kappa^{q-b}d_{\max}^{2(N-1)}\O{1} 
.\label{contribution}
\end{eqnarray}
The $\O{1}$ is due to the fact that among long cycles,  bridges can be made around $i$- or $j-$vertices. 
The case where $N=0$ corresponds to the case where the graph is a tree of simple cycles. This yields the first term in  (\ref{contribution}).
In the last line of (\ref{contribution}), we consider the contributions of those graphs such that $b=I_i+I_j+1-N$ for some $N>0.$ 
Now the case where $N=1$ yields the following contribution since $b=I_i+I_j$:
\begin{multline}(1+o(1))\sum_{ I_i, I_j} \mathcal{A}(q, I_i,I_j,I_i+I_j) \psi^{-I_j} \phi^{I_j} \theta_1^{I_i+I_j}\left(\frac{\theta_3\kappa}{\psi}\right)^{q-I_i-I_j}\\
=(1+o(1))\sum_{ I_i, I_j} \mathcal{A}(q, I_i,I_j,I_i+I_j)
\gamma^{I_j} \theta_1^{I_i+I_j}\left(\frac{\theta_3\kappa}{\psi}\right)^{q-I_i-I_j}.\label{(e)}
\end{multline}
Note that in this case the two contributions of bridges around $i$-indices and $j$-indices contribute. This is the reason for the $(1+o(1))$ correction term.
%Now the number of distinct $i$ vertices in the non necessarily connected graph of simple cycles is $b-I_i$. Note that the origins of these subgraphs of simple cycles may be on the long cycle: they are thus counted when counting the number of distinct vertices along the long cycle.

We are going to show that  the only graphs that indeed contribute for high moments $q\sim (\log n_1)^{1+\alpha_1}$ due to the power of $n_0$ are such that $N=1$  which means there is a unique cycle of length greater than 2. 
To this aim we relate their contribution to (\ref{contribution}) to another model of random matrices. 
\paragraph{Encoding an admissible graph with one or more long cycles}

First, one can see that examining the contribution of admissible graphs with some long cycles is equivalent to the following moment problem: evaluating the contribution to the expectation from edges which can be seen twice (and are in the fat trees) and contribute by $\theta_1$ to the expected value or odd edges seen only once in some long cycle  and contribute each by $\sqrt{\theta_3 \kappa/ \psi}$ at most. Such graphs also arise when one computes the moments of an information-plus-noise sample covariance matrix $M'=\frac{1}{m}(\tilde Z+\alpha J)(\tilde Z+\alpha J)^\top$, where we fix
%one simply considers entries which are non centered and have expectation 
$\alpha=\sqrt{\theta_3\kappa/ \psi n_1}$. Indeed, edges in this latter case arise either twice (or more) or a single time (for $J$ entries). It follows from  \cite{Soshnikov} or \cite{FePe}  that the typical graphs in $\mathbb{E}\text{Tr}(M')^q$ are exactly those for which edges arise once or twice but not more if $q\ll \sqrt{m}$.  Edges arising only once necessarily appear inside  long cycles and correspond to $J$ entries. This follows from the fact that when computing a trace, each graph contributing to the expectation comes from the simple cycle of length $2q$ on which one makes some identification of the vertices. It is also shown that edges appearing twice correspond to $\tilde Z$ entries typically.
%Each vertex has then an even degree: as a consequence edges arising only once necessarily belong to a long cycle. Other edges necessarily arise at least twice by independence and the fact $\tilde Z$ entries are centered.

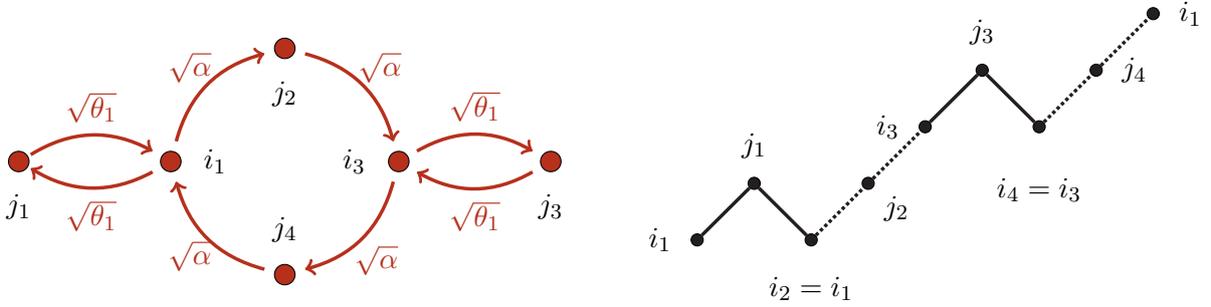
\begin{figure}[!ht]
%	\begin{minipage}{.3\linewidth}
%		\centering
%		\begin{tikzpicture}
%			\draw[-, BrickRed, line width=.12em] (0,0) circle (1.5cm);
%			\foreach \a in {1,2,...,8}{
%				\node[draw=none, inner sep=0pt, minimum size = .54cm] (\a+1) at (90+360/8*\a:1.5cm) {};
%				\node[fill=BrickRed, inner sep=0pt, minimum size=.27cm] (\a) at (90+360/8*\a:1.5cm) {};
%				
%			}
%		\node[draw=none] at (90+360/4:2cm) {$i_1$};
%		\node[draw=none] at (90+360/8:2cm) {$j_1$};
%		\node[draw=none] at (90+360*8/8:2cm) {$i_2$};
%		\node[draw=none] at (90+360/8*3:2cm) {$j_4$};
%		\node[draw=none] at (90+360/8*7:2cm) {$\dots$};
%		
%		\draw[<->, ForestGreen, line width=.15em] (2+1) edge[bend right=60] (8+1);
%		\draw[<->, ForestGreen, line width=.15em] (4+1) edge[bend left=60] (6+1);
%		\end{tikzpicture}
%	\end{minipage}
	\begin{minipage}{.49\linewidth}
		\centering
		\begin{tikzpicture}
			\node[draw=none, inner sep= 0pt, minimum size=.35cm, label=below:$j_1$] (0+1) at (-3.5,0) {};
			\node[fill=BrickRed, inner sep= 0pt, minimum size=.27cm] (0) at (-3.5,0) {};
			\foreach \a in {1,2,3,4}{
				\node[draw=none, inner sep = 0pt, minimum size=.54cm] (\a+1) at (90+360/4*\a:1.5cm) {};
				\node[fill = BrickRed, inner sep=0pt, minimum size=.27cm] (\a) at (90+360/4*\a:1.5cm) {};
			}
			\node[draw=none, inner sep= 0pt, minimum size=.54cm] (5+1) at (3.5,0) {};
			\node[fill=BrickRed, inner sep= 0pt, minimum size=.27cm] (5) at (3.5,0) {};
			
			\node[draw=none, inner sep=0pt, minimum size=.35cm, label=right:$i_1$] at (-1.5,0) {};
			
			\node[draw=none, inner sep=0pt, minimum size=.35cm, label=below:$j_2$] at (0,1.5) {};
			
			\node[draw=none, inner sep=0pt, minimum size=.35cm, label=above:$j_4$] at (0,-1.5) {};
			
			\node[draw=none, inner sep=0pt, minimum size=.35cm, label=left:$i_3$] at (1.5,0) {};
			
			\node[draw=none, inner sep=0pt, minimum size=.35cm, label=below:$j_3$] at (3.5,0) {};
			
			\draw[->, BrickRed, line width=.12em] (0+1) edge[bend left=30] node[draw=none, midway, inner sep=0pt, minimum size=.05cm, label={[label distance = -.7em]above:\textcolor{BrickRed}{$\sqrt{\theta_1}$}}] {} (1+1);
			
			\draw[->, BrickRed, line width=.12em] (1+1) edge[bend left=30] node[draw=none, midway, inner sep=0pt, minimum size=.05cm, label={[label distance = -.7em]below:\textcolor{BrickRed}{$\sqrt{\theta_1}$}}] {} (0+1);
			
			\draw[->, BrickRed, line width=.12em] (1+1) edge[bend left] node[draw=none, midway, inner sep=0pt, minimum size=.05cm, label={[label distance = -.7em]135:\textcolor{BrickRed}{$\sqrt{\alpha}$}}] {} (4+1);
			
			\draw[->, BrickRed, line width=.12em] (4+1) edge[bend left] node[draw=none, midway, inner sep=0pt, minimum size=.05cm, label={[label distance = -.7em]45:\textcolor{BrickRed}{$\sqrt{\alpha}$}}] {} (3+1);
			
			\draw[->, BrickRed, line width=.12em] (3+1) edge[bend left] node[draw=none, midway, inner sep=0pt, minimum size=.05cm, label={[label distance = -.7em]335:\textcolor{BrickRed}{$\sqrt{\alpha}$}}] {} (2+1);
			
			\draw[->, BrickRed, line width=.12em] (2+1) edge[bend left] node[draw=none, midway, inner sep=0pt, minimum size=.05cm, label={[label distance = -.7em]225:\textcolor{BrickRed}{$\sqrt{\alpha}$}}] {} (1+1);
			
			\draw[->, BrickRed, line width=.12em] (3+1) edge[bend left=30] node[draw=none, midway, inner sep=0pt, minimum size=.05cm, label={[label distance = -.7em]above:\textcolor{BrickRed}{$\sqrt{\theta_1}$}}] {} (5+1);
			
			\draw[->, BrickRed, line width=.12em] (5+1) edge[bend left=30] node[draw=none, midway, inner sep=0pt, minimum size=.05cm, label={[label distance = -.7em]below:\textcolor{BrickRed}{$\sqrt{\theta_1}$}}] {} (3+1);
		\end{tikzpicture}
	\end{minipage}
%	\hspace{2em}
	\begin{minipage}{.49\linewidth}
		\centering
		\begin{tikzpicture}
			\node[fill=Black, inner sep=0pt, minimum size=.16cm, label=left:$i_1$] (1) at (0,0) {};
			\node[fill=Black, inner sep=0pt, minimum size=.16cm, label=above:$j_1$] (2) at (.75,.75) {};
			\node[fill=Black, inner sep=0pt, minimum size=.16cm, label={[label distance = -.5em]below:${i_2=i_1}$}] (3) at (1.5,0) {};
			\node[fill=Black, inner sep=0pt, minimum size=.16cm, label=335:$j_2$] (4) at (2.25,.75) {};
			\node[fill=Black, inner sep=0pt, minimum size=.16cm, label=left:$i_3$] (5) at (3,1.5) {};
			\node[fill=Black, inner sep=0pt, minimum size=.16cm, label=above:$j_3$] (6)
			 at (3.75,2.25) {};
			\node[fill=Black, inner sep=0pt, minimum size=.16cm, label=below:${i_4=i_3}$] (7) at (4.5,1.5) {};
			\node[fill=Black, inner sep=0pt, minimum size=.16cm, label=right:$j_4$] (8) at (5.25,2.25) {};
			\node[fill=Black, inner sep=0pt, minimum size=.16cm, label=right:$i_1$] (9) at (6,3) {};
			
			\draw[-,Black,line width = .12em] (1) edge (2);			
			\draw[-,Black,line width = .12em] (2) edge (3);
			\draw[-,Black,line width = .12em] (5) edge (6);
			\draw[-,Black,line width = .12em] (6) edge (7);
			
			\draw[densely dotted,Black,line width = .12em] (3) edge (4);
			\draw[densely dotted,Black,line width = .12em] (4) edge (5);
			\draw[densely dotted,Black,line width = .12em] (7) edge (8);
			\draw[densely dotted,Black,line width = .12em] (8) edge (9);
		\end{tikzpicture}
	\end{minipage}
	\caption{A red admissible graph with a long cycle can be encoded in a bipartite path with marked edges as in the study of information-plus-noise models. The dotted lines contribute $\sqrt{\alpha}$ while the solid lines contribute $\sqrt{\theta_1}$. In the case $\gamma=1$, note that the path is not bipartite as $i$ and $j$-indices play the same role and we recover the case of \cite{FePe}}
\end{figure}
\paragraph{}When  $\gamma=1$ (more precisely if $\vert n_1/m-1\vert \ll (\ln m)^{-1}$),  this combinatorial question is the same as computing  the tracial moments of a Wigner matrix $W_m/\sqrt m$ of size $m\times m$ whose entries have variance $\theta_1$ and with non centered entries whose expected value is $\sqrt{\theta_3 \kappa/\psi m}.$ Indeed when $\gamma=1$, the Marchenko--Pastur distribution is simply the law of a squared Wigner random variable. It has been shown that if $\sqrt{\theta_3 \kappa/\psi}> \sqrt{\theta_1}$, the largest eigenvalue of $W_m/\sqrt m$ then behaves as 
$$\rho\left( \sqrt{\frac{\theta_3 \kappa}{\psi}}\right)\quad \text{ where }\quad\rho(x)=x+\frac{ \sqrt{\theta_1}}{x}.$$ In particular, there is an eigenvalue exiting the bulk of the spectrum in the large $m$ limit
and one has that $\E\text{Tr}\left( \frac{W_m}{\sqrt m}\right)^q= \rho(\sqrt{\theta_3 \kappa/\psi})^q (1+o(1)).$ Additionally, it is proved in \cite{FePe} that odd edges arise typically in a single long cycle (up to an error in the order of $q^2/n_0$).

One can be even more precise. Recall the definition of $\kappa_w$ and $\kappa_x$ from \eqref{kappa}. If $$\sqrt{\frac{\theta_3 \kappa_w}{\psi}}> \sqrt{\theta_1}\quad\text{and}\quad\sqrt{\frac{\theta_3 \kappa_x}{\psi}}> \sqrt{\theta_1},$$ then 
the contribution from graphs with one long cycle, (and admissible red graphs with one long cycle) is then from \cite{FePe}*{Theorem 4.1}
\[ 
(1+o(1))\left (\rho\left(\sqrt{\frac{\theta_3 \kappa_w}{\psi}}\right)^{2q}+ \rho\left(\sqrt{\frac{\theta_3 \kappa_x}{\psi}}\right)^{2q}\right ).\]

\paragraph{}When $\gamma>1$, one can see the edges of the long cycle as odd edges each contributing by a factor $\sqrt{\theta_3\kappa/\psi}$. Thus  the contribution of an admissible graph with a single long cycle is equal to its contribution to the moment  $m^{-q}\mathbb{E} \text{Tr}T^q $ where $T$ is the $n_1\times n_1$ information plus noise matrix
$$T=\frac{1}{m}\left (\sqrt{\theta_1}\tilde Z+\sqrt{\frac{\theta_3 \kappa}{n_0}} J\right )\left (\sqrt{\theta_1}\tilde Z+\sqrt{\frac{\theta_3 \kappa}{n_0} } J\right )^\top.$$ 
Observe that $\psi n_1\sim n_0$ so that the result is consistent with the case where $\gamma=1$.
Indeed, for such a model, a long cycle of length $2c$ labeled with $J$ edges only is weighted with normalization 
$$\frac{n_1^c m^c}{m^{c}}\left ( \sqrt{\frac{\theta_3\kappa}{n_0}} \right)^{2c}= \left(\frac{n_1}{n_0}\right)^c (\theta_3\kappa)^c\sim \left(\sqrt{\frac{\theta_3\kappa}{\psi}}\right)^{2c}$$
which does correspond to the weight in our nonlinear ensemble. 

Note that even though the tracial moments of the two models match, actually counting the number of these admissible graphs and understand the position of the largest eigenvalue using high moments is difficult and has not been done in the literature to our knowledge. Adding a bipartite structure to the analysis from \cite{FePe} makes the computations too difficult.  However, the position of the largest eigenvalue of such information-plus-noise models have been studied differently and we refer to \cite{Capitaine} for the precise location and behavior of this largest eigenvalue. 

We call $\rho'(\theta_3 )$ the top eigenvalue of the above information-plus-noise matrix ensemble and we see that since $q$ goes to infinity in such a way that $q^2\ll n_0$, the contribution of admissible graphs with a long cycle from \cite{Capitaine} is 
$$(1+o(1))\rho'(\theta_3 )^q.$$ This follows from the fact that contributions from paths with more than one long cycle of odd edges induces a cost in the order of $q^2n_0^{-1}$ and is thus negligible (as for $\gamma=1$ by simply modifying the arguments of \cite{FePe}). As a result the typical contribution to $m^{-q}\mathbb{E} \text{Tr}T^q $ comes from paths with one single long cycle at most: the largest eigenvalue then separates when the contribution of all graphs with a single long cycle exceeds that of trees of simple cycles.

To conclude the argument, one has to show that the contribution of admissible graphs with more than one long cycle also yields a negligible contribution for the nonlinear model. To do so, we are going to define a mapping between admissible graphs with one long cycle of length $2L\leqslant 8$ and admissible graphs with two long cycles. Consider a graph with one long cycle, it can be represented as a cycle of certain length $2L$ with a tree of cycle of length 2 attached to each vertex of the cycle. 

To construct an admissible graph with two cycles from a given one with a unique long cycle, first consider two vertices say $v_1$ and $v_2$ on the long cycle such that we have $\d(v_1,v_2)=2k$ for some $k\geqslant 2$ where $\d$ denotes the minimum number of edges in the long cycle between $v_1$ and $v_2.$ At those chosen vertices, we consider the tree of cycles of length 2 attached to it and we choose a vertex in each tree (note that we can choose again $v_1$ or $v_2$). All in all we can bound the number of choices by $q^4$ since there are at most $q$ vertices in the graph. This is illustrated in Figure \ref{fig:choice}.
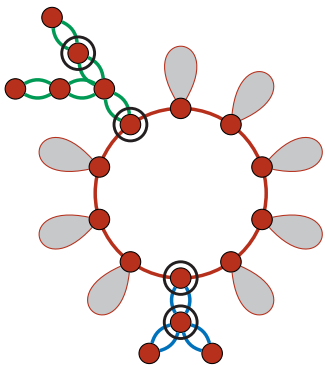
\begin{figure}[!ht]
	\centering
	\begin{tikzpicture}[scale=.9]
		\draw[-, BrickRed, line width = .12em] (0,0) circle (1.25cm);
		
		\foreach \a in {1,2,...,10}{
			\node[fill=BrickRed, inner sep = 0pt, minimum size=.27cm] (\a) at (360/4+\a*360/10: 1.25cm) {};
		}
		
		\foreach \a in {2,3,...,10}{
			\ifthenelse{\a=5}{}{
				\begin{scope}[on background layer]
					\path [draw, BrickRed, fill=Gray!50] (\a) to [loop above, looseness=10, min distance=1.2cm, out=60+\a*36, in=120+\a*36] (\a);
			\end{scope} }
		}
		
		\node[fill = BrickRed, inner sep=0pt, minimum size=.27cm] (11) at (126:1.9cm) {};
		\node[fill = BrickRed, inner sep=0pt, minimum size=.27cm] (12) at (126:2.55cm) {};
		\node[fill = BrickRed, inner sep=0pt, minimum size=.27cm] (13) at (126:3.2cm) {};
		\node[fill = BrickRed, inner sep=0pt, minimum size=.27cm] (14) at ($(11)+(180:.65cm)$) {};
		\node[fill = BrickRed, inner sep=0pt, minimum size=.27cm] (15) at ($(11)+(180:1.3cm)$) {};
		
		\node[fill = BrickRed, inner sep=0pt, minimum size=.27cm] (16) at (270:1.9cm) {};
		\node[fill = BrickRed, inner sep=0pt, minimum size=.27cm] (17) at ($(16)+(225:.65cm)$) {};
		\node[fill = BrickRed, inner sep=0pt, minimum size=.27cm] (18) at ($(16)+(315:.65cm)$) {};
		
		\draw[-,ForestGreen, line width = .12em] (1) edge[bend left] (11);
		\draw[-,ForestGreen, line width = .12em] (1) edge[bend right] (11);
		
		\draw[-,ForestGreen, line width = .12em] (11) edge[bend left] (12);
		\draw[-,ForestGreen, line width = .12em] (11) edge[bend right] (12);
		
		\draw[-,ForestGreen, line width = .12em] (12) edge[bend left] (13);
		\draw[-,ForestGreen, line width = .12em] (12) edge[bend right] (13);
		
		\draw[-,ForestGreen, line width = .12em] (11) edge[bend left] (14);
		\draw[-,ForestGreen, line width = .12em] (11) edge[bend right] (14);
		
		\draw[-,ForestGreen, line width = .12em] (14) edge[bend left] (15);
		\draw[-,ForestGreen, line width = .12em] (14) edge[bend right] (15);
		
		\foreach \a in {5,17,18}{
			\draw[-,RoyalBlue, line width = .12em] (16) edge[bend left] (\a);
			\draw[-,RoyalBlue, line width = .12em] (16) edge[bend right] (\a);
		}
		
		\draw[-, Black, line width = .1em] (1.center) circle (.24cm);
		\draw[-, Black, line width = .1em] (5.center) circle (.24cm);
		\draw[-, Black, line width = .1em] (12.center) circle (.24cm);
		\draw[-, Black, line width = .1em] (16.center) circle (.24cm);
	\end{tikzpicture}
	\caption{Example of choosing two vertices on the unique long cycle and one vertex each in the tree attached to those vertices. Here we see that the long cycle is of length $L=10>8$ and we have $\d(v_1,v_2)=4.$}
	\label{fig:choice}
\end{figure}

We create two cycles by identifying $v_1$ and $v_2$, note that since they are chosen so that $\d(v_1,v_2)$ is even, the two cycles are of even length and is consistent with the bipartite structure of the graph. Both trees that were attached to $v_1$ and $v_2$ are now attached to the same vertex which links the two long cycles. In particular, it is easy to see that one cycle is of length $\d(v_1,v_2)$ and the other is of size $L-\d(v_1,v_2).$ This is illustrated in Figure \ref{fig:glue}.

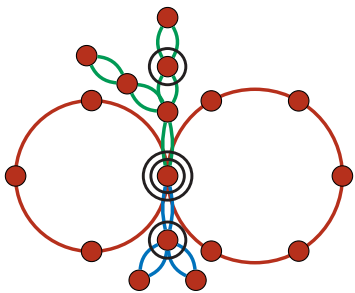
\begin{figure}[!ht]
	\centering
	\begin{tikzpicture}
		\draw[-,BrickRed, line width = .12em] (0,0) circle (1cm);
		\draw[-,BrickRed, line width = .12em] (2.15,0) circle (1.15cm);
		
		\foreach \a in {1,2,3,4}{
			\node[fill = BrickRed, inner sep = 0pt, minimum size=.27cm] (\a) at (\a*90:1cm) {};
		}
		\foreach \a in {1,2,3,4,5}{
			\node[fill = BrickRed, inner sep = 0pt, minimum size=.27cm] (\a+1) at ($(2.15,0)+(180+\a*360/6:1.15cm)$) {};
		}
		
		\node[fill = BrickRed, inner sep=0pt, minimum size=.27cm] (10) at ($(4)+(90:.85cm)$) {};
		\node[fill = BrickRed, inner sep=0pt, minimum size=.27cm] (11) at ($(4)+(90:1.45cm)$) {};
		\node[fill = BrickRed, inner sep=0pt, minimum size=.27cm] (12) at ($(4)+(90:2.1cm)$) {};
		\node[fill = BrickRed, inner sep=0pt, minimum size=.27cm] (13) at ($(10)+(145:.65cm)$) {};
		\node[fill = BrickRed, inner sep=0pt, minimum size=.27cm] (14) at ($(10)+(145:1.3cm)$) {};
		
		\node[fill = BrickRed, inner sep=0pt, minimum size=.27cm] (16) at ($(4)+(270:.85cm)$) {};
		\node[fill = BrickRed, inner sep=0pt, minimum size=.27cm] (17) at ($(16)+(235:.65cm)$) {};
		\node[fill = BrickRed, inner sep=0pt, minimum size=.27cm] (18) at ($(16)+(305:.65cm)$) {};
		
		\draw[-,ForestGreen, line width = .12em] (4) edge[bend left=12] (10);
		\draw[-,ForestGreen, line width = .12em] (4) edge[bend right=12] (10);
		
		\draw[-,ForestGreen, line width = .12em] (10) edge[bend left] (11);
		\draw[-,ForestGreen, line width = .12em] (10) edge[bend right] (11);
		
		\draw[-,ForestGreen, line width = .12em] (11) edge[bend left] (12);
		\draw[-,ForestGreen, line width = .12em] (11) edge[bend right] (12);
		
		\draw[-,ForestGreen, line width = .12em] (10) edge[bend left] (13);
		\draw[-,ForestGreen, line width = .12em] (10) edge[bend right] (13);
		
		\draw[-,ForestGreen, line width = .12em] (13) edge[bend left] (14);
		\draw[-,ForestGreen, line width = .12em] (13) edge[bend right] (14);

		\foreach \a in {17,18}{
			\draw[-,RoyalBlue, line width = .12em] (16) edge[bend left] (\a);
			\draw[-,RoyalBlue, line width = .12em] (16) edge[bend right] (\a);
		}
		
		\draw[-,RoyalBlue, line width = .12em] (16) edge[bend left=12] (4);
		\draw[-,RoyalBlue, line width = .12em] (16) edge[bend right=12] (4);
		
		\draw[-, Black, line width = .1em] (4.center) circle (.22cm);
		\draw[-, Black, line width = .1em] (4.center) circle (.32cm);
		\draw[-, Black, line width = .1em] (11.center) circle (.24cm);
		\draw[-, Black, line width = .1em] (16.center) circle (.24cm);
	\end{tikzpicture}
	\caption{We identify the two vertices chosen in the long cycle which create two cycles here of length 4 and 6. For simplicity, we removed the presence of trees of cycle of length 2 attached to every other vertex since they do not play a role in this mapping.}
	\label{fig:glue}
\end{figure}

Finally, we separate the two cycles by identifying the two vertices chosen in each attached tree of cycles to form a new tree between the two long cycles. Note that if we have chosen $v_1$ and $v_2$ themselves on the trees, the two cycles would be attached to themselves as in Figure \ref{fig:glue}. Note that there is a choice of which long cycle to put $v_1$ and $v_2$ which gives a factor $2$. This is illustrated in Figure \ref{fig:step3}.

\begin{figure}
	\centering
	\begin{tikzpicture}
		\draw[-,BrickRed, line width = .12em] (0,0) circle (1cm);
		
		\foreach \a in {1,2,3,4}{
			\node[fill = BrickRed, inner sep = 0pt, minimum size=.27cm] (\a) at (\a*90:1cm) {};
		}
		
		\node[fill = BrickRed, inner sep=0pt, minimum size=.27cm] (5) at ($(4)+(0:.65cm)$) {};
		\node[fill = BrickRed, inner sep=0pt, minimum size=.27cm] (6) at ($(4)+(0:1.3cm)$) {};
		
		\node[fill = BrickRed, inner sep=0pt, minimum size=.27cm] (7) at ($(5)+(90:.65cm)$) {};
		\node[fill = BrickRed, inner sep=0pt, minimum size=.27cm] (8) at ($(5)+(90:1.3cm)$) {};
		
		\node[fill = BrickRed, inner sep=0pt, minimum size=.27cm] (9) at ($(6)+(235:.65cm)$) {};
		\node[fill = BrickRed, inner sep=0pt, minimum size=.27cm] (10) at ($(4)+(0:1.95cm)$) {};
		\node[fill = BrickRed, inner sep=0pt, minimum size=.27cm] (11) at ($(6)+(305:.65cm)$) {};
		\node[fill = BrickRed, inner sep=0pt, minimum size =.27cm] (12) at ($(6)+(90:.65cm)$) {};
		
		\draw[-,BrickRed, line width = .12em] (4.1,0) circle (1.15cm);
		
		\foreach \a in {1,2,3,4,5}{
			\node[fill = BrickRed, inner sep = 0pt, minimum size=.27cm] (\a+1) at ($(4.1,0)+(180+\a*360/6:1.15cm)$) {};
		}
		
		\draw[-,ForestGreen, line width = .12em] (4) edge[bend left] (5);
		\draw[-,ForestGreen, line width = .12em] (4) edge[bend right] (5);
		
		\draw[-,ForestGreen, line width = .12em] (5) edge[bend left] (6);
		\draw[-,ForestGreen, line width = .12em] (5) edge[bend right] (6);
		
		\draw[-,ForestGreen, line width = .12em] (5) edge[bend left] (7);
		\draw[-,ForestGreen, line width = .12em] (5) edge[bend right] (7);
		
		\draw[-,ForestGreen, line width = .12em] (7) edge[bend left] (8);
		\draw[-,ForestGreen, line width = .12em] (7) edge[bend right] (8);
		
		\draw[-,ForestGreen, line width = .12em] (6) edge[bend left] (12);
		\draw[-,ForestGreen, line width = .12em] (6) edge[bend right] (12);
		
		\draw[-,RoyalBlue, line width = .12em] (6) edge[bend left] (9);
		\draw[-,RoyalBlue, line width = .12em] (6) edge[bend right] (9);
		
		\draw[-,RoyalBlue, line width = .12em] (6) edge[bend left] (10);
		\draw[-,RoyalBlue, line width = .12em] (6) edge[bend right] (10);
		
		\draw[-,RoyalBlue, line width = .12em] (6) edge[bend left] (11);
		\draw[-,RoyalBlue, line width = .12em] (6) edge[bend right] (11);
		
		\foreach \a in {1,2,...,5}{
			\begin{scope}[on background layer]
				\path [draw, BrickRed, fill=Gray!50] (\a+1) to [loop above, looseness=10, min distance=1.2cm, out=150+\a*60, in=210+\a*60] (\a+1);
			\end{scope} 
		}
		
		\foreach \a in {1,2,3}{
			\begin{scope}[on background layer]
				\path [draw, BrickRed, fill=Gray!50] (\a) to [loop above, looseness=10, min distance=1.2cm, out=30+\a*90, in=330+\a*90] (\a);
			\end{scope} 
		}

		\draw[-, Black, line width = .1em] (4.center) circle (.24cm);
		\draw[-, Black, line width = .1em] (10.center) circle (.24cm);
		\draw[-, Black, line width = .1em] (6.center) circle (.22cm);
		\draw[-, Black, line width = .1em] (6.center) circle (.32cm);
	\end{tikzpicture}
	\caption{The two trees that were attached to $v_1$ and $v_2$ are now glued through the identification of the two vertices chosen initially. We could have chosen to attach the green tree to the cycle of the right and the blue one to the left instead.}
	\label{fig:step3}
\end{figure}
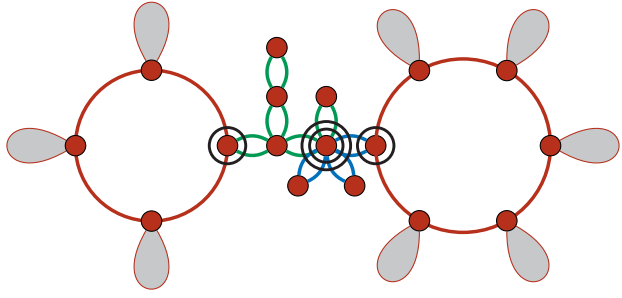

This mapping is onto on the graphs with two long cycles but not one-to-one as we overcount graphs when we glue the trees at the final step since there are many ways to glue two trees to obtain the same tree. Note also that this construction create an additonal identification between vertices (either an $i$ or a $j$ identifications) which is consistent with the fact that we have gain one cycle since the total number of cycle is $I_i+I_j+1$, finally we can bound
\[
\sum_{I_i,I_j=0}^q \left(\A(q, I_i,I_j+1,I_i+I_j)+\A(q,I_i+1,I_j,I_i+I_j)\right)
\leqslant Cq^4\sum_{I_i,I_j=0}^q \A(q,I_i,I_j,I_i+I_j).
\]
In particular, if we consider the contribution (with the corresponding weights) of the graphs with two long cycles we have
\begin{multline*}
\frac{\psi}{n_0}\sum_{I_i,I_j=0}^q\A(q,I_i,I_j,I_i+I_j-1)\theta_1^{I_i+I_j-1}\left(\frac{\theta_3\kappa}{\psi}\right)^{q-I_i-I_j+1}\gamma^{I_j}
\\\leqslant
C\frac{\theta_3\kappa q^4}{\theta_1 n_0}
\sum_{I_i,I_j=0}^q\A(q,I_i,I_j,I_i+I_j)\theta_1^{I_i+I_j}\left(\frac{\theta_3\kappa}{\psi}\right)^{q-I_i-I_j}\gamma^{I_j}.
\end{multline*}
We see that the sum in the right hand side is the contribution of graphs with a unique long cycle and thus we obtain that the contribution from graphs with two long cycles is smaller by a factor of $Cq^4n_0^{-1}=o(1)$ for $q\leqslant (\log n_1)^{1+\alpha}$. 

If we consider graphs with a given number of long cycles $N$, we can perform the same construction to bound the contribution by the one from admissible graphs with $N-1$ long cycles and recursively to the contribution of admissible graphs with a unique long cycle. The construction is similar: we first choose a long cycle of length bigger than $8$ among the $N$ choices (we can bound crudely this number of choices by $q$), we then choose two vertices with an even distance greater than $4$ within the cycle (this gives a factor of $q^2$), since the graph is admissible at both these vertices we have an admissible graph attached to it instead of a tree of cycles of length $2$ and we still choose a given vertex in each admissible graph (this gives a factor of $q^2$) we then do the same final step by identifying the two vertices from the long cycle to create two long cycles and then identify the two vertices from the two admissible graphs and attach it between the two created long cycles. This is similar as the process illustrated in Figures \ref{fig:choice}, \ref{fig:glue}, and \ref{fig:step3}.

Overall we see that we can bound the number of choices by $Cq^5$ during this procedure. This gives,
\begin{align*}
\sum_{I_i,I_j=0}^q
\sum_{N=2}^{I_i+I_j+1}
&\left( 
\frac{\psi}{n_0}
\right)^{N-1}
\A(q,I_i,I_j,I_i+I_j+1-N)\theta_1^{I_i+I_j+1-N}\left(\frac{\theta_3\kappa}{\psi}\right)^{q-I_i-I_j-1+N}\gamma^{I_j}
\\
&\leqslant
\sum_{I_i,I_j=0}^q
\left(\sum_{N=2}^{I_i+I_j+1}
\left(C\frac{\theta_3\kappa q^5}{\theta_1 n_0}\right)^{N-1}
\right)
\A(q,I_i,I_j,I_i+I_j)\theta_1^{I_i+I_j}\left(\frac{\theta_3\kappa}{\psi}\right)^{q-I_i-I_j}\gamma^{I_j}
\\
&\leqslant 
C\frac{\theta_3\kappa q^6}{\theta_1 n_0}
\sum_{I_i,I_j=0}^q
\A(q,I_i,I_j,I_i+I_j)\theta_1^{I_i+I_j}\left(\frac{\theta_3\kappa}{\psi}\right)^{q-I_i-I_j}\gamma^{I_j}
\end{align*}
where in the last inequality we bounded the number of long cycles by $q$.

Finally, we see that, for both cases $\gamma=1$ or $\gamma >1$, rewriting the sum in terms of number of cycles of length $2$ instead of long cycles,
\begin{equation}\sum_{ I_i, I_j=0}^q\sum_{b=0}^{I_i+I_j-1} 
\mathcal{A}(q, I_i,I_j,b) \phi^{I_j}\psi^{-b+I_i}\theta_1^{b}\left(\frac{\theta_3\kappa}{\psi}\right)^{q-b} d_{\max}^{2(N-1)}n_0^{b-I_i-I_j}\leqslant \frac{\Vert T\Vert^q}{n_0}(Cq^7+o(1)).
\end{equation}
The above follows from majorizing $d_{\max}\leq q$ and comparing the contribution in each ensemble for a given number of long cycles. 

The analysis of the non admissible graphs follows the same lines as in \cite{BePe} and yields a contribution which is negligible with respect to that of the admissible graphs. The arguments are skipped here. This finishes the proof of Proposition \ref{prop:even}.\end{proof}

%By using Markov's inequality, we deduce that, for any $\epsilon >0$, the largest eigenvalue does not exceed $\rho'(\theta_3)(1+\epsilon)$. Because $q$ can grow as large as $(\log n_1)^{1+\alpha}$ we deduce that $\lambda_1\leqslant \rho'(\theta_3)(1-\epsilon)$.

\section{Proof of main results for general activation functions \label{sec:5}}
This section is devoted to generalizing the two previous sections to general activation functions. When $\theta_3(f)=0$, we saw the example of an odd polynomial, we generalize to general activation functions by examining the impact of $\theta_3(f)=0$ on the combinatorics involving even monomials. We show that long cycles in admissible graphs can only contribute $\theta_2(f)$ in this case. When $\theta_2(f)=0$, the contribution of odd monomials is similar as in \cite{BePe} and only graphs with a single long cycle contribute to the spectral moments. When $\theta_2(f)\not=0$, the limiting e.e.d. is no longer the Marchenko--Pastur distribution and as in the case where $f$ is odd, the moments of the spectral distribution is a combined contribution of all admissible graphs with long cycles. In addition, when $\theta_3(f)\not=0$ and since $q\sim (\log n_1)^{1+\alpha}$, one cannot neglect the contribution from a long cycle with bridges added to the matchings of the blue indices.

\subsection{Proofs of Theorems \texorpdfstring{\ref{theo:edge} and \ref{theo:sep_easy}}{2.3 and 2.4}} 
\begin{proof}[Proof of Theorem \ref{theo:edge}]
We now need to consider the contribution of even monomials in the combinatorics from Section \ref{sec:3} since while odd functions satisfy $\theta_3(f)=0$ (since $f''$ is odd), the set of such functions is larger. To do so, we can consider the analysis done in Section \ref{sec:4} which shows that the leading contribution stems from matchings giving a contribution involving $\theta_3(f)$. We now focus on the case of an even monomial.  The contribution of {non admissible graphs} is controlled similarly as for odd monomials, since this contribution does not depend on the function but only on the initial matching and is developed below. The contribution of {admissible} graphs is similar but the matchings are different. 

Firstly, remember that the function is centered so that $\E[f(\mathcal{N}(0,1)]=0$. In particular, when considering a matching within a red cycle, this prevents performing a perfect matching in any niche between the $k$ (which is even) blue vertices. At least one additional identification must be done in each niche, and it is seen in Section \ref{sec:4} that the leading contribution for a cycle of length $q$ is of order $n_1^{-1}\theta_3(P_k)^q$. This is illustrated in Figure \ref{fig:matchadmisseven}. Note that if $q\sim (\log n_1)^{1+\alpha_1}$ and $\theta_3$ is large enough, this contribution can become leading as this is only compensated by $n_1^{-1}$. This is the reason why the largest eigenvalue can escape the bulk of the spectrum. However, since we suppose that $\theta_3(f)=0$, the contribution of such matchings vanishes as $k$ grows with $n_1$.

We need to look at the contributions of lower order. For such a contribution to possibly be leading when $q\sim (\log n_1)^{1+\alpha_1},$ it needs to involve all niches simultaneously so that the combinatorial factor is to the power $q$ (the length of the cycle). However, we see that such matchings are compensated by $n_1^{-q}$ and thus subleading even when considering large moments. These matchings are described  in Figure \ref{fig:subleadeven}. 
\begin{figure}[!ht]
	\centering
	\begin{minipage}{.45\linewidth}
		\centering
		\begin{tikzpicture}
			\draw[-,BrickRed,line width = .12em] (0,0) circle (1cm);
			\node[fill=BrickRed,inner sep=0pt, minimum size=.27cm] (2) at (1,0) {};
			\node[fill=BrickRed,inner sep=0pt, minimum size=.27cm] at (-1,0) {};
			\node[fill=BrickRed,inner sep=0pt, minimum size=.27cm] (1) at (0,1) {};
			\node[fill=BrickRed,inner sep=0pt, minimum size=.27cm] at (0,-1) {};
			\foreach \a in {1,2,...,19}{
			\ifthenelse{\a=5 \OR \a=10 \OR \a=15}{}{\node[fill=RoyalBlue, inner sep = 0pt, minimum size=.12cm] (\a+1) at (360/4+\a*360/20: 1cm) {};}
			}
			
			\draw[-, RoyalBlue, line width=.12em] (4+1) edge[bend left=60]  (6+1);
			\draw[-, RoyalBlue, line width=.12em] (3+1) edge[bend left=60]  (7+1);
			\draw[-, RoyalBlue, line width=.12em] (6+1) edge[bend left=60]  (7+1);
			\draw[-, RoyalBlue, line width=.12em] (3+1) edge[bend left=60]  (4+1);
			
			\draw[-, RoyalBlue, line width=.12em] (9+1) edge[bend left=60]  (11+1);
			\draw[-, RoyalBlue, line width=.12em] (8+1) edge[bend left=60]  (12+1);
			\draw[-, RoyalBlue, line width=.12em] (11+1) edge[bend left=60]  (12+1);
			\draw[-, RoyalBlue, line width=.12em] (8+1) edge[bend left=60]  (9+1);
			
			\draw[-, RoyalBlue, line width=.12em] (14+1) edge[bend left=60]  (16+1);
			\draw[-, RoyalBlue, line width=.12em] (13+1) edge[bend left=60]  (17+1);
			\draw[-, RoyalBlue, line width=.12em] (16+1) edge[bend left=60]  (17+1);
			\draw[-, RoyalBlue, line width=.12em] (13+1) edge[bend left=60]  (14+1);
			
			\draw[-, RoyalBlue, line width=.12em] (19+1) edge[bend left=60]  (1+1);
			\draw[-, RoyalBlue, line width=.12em] (18+1) edge[bend left=60]  (2+1);
			\draw[-, RoyalBlue, line width=.12em] (1+1) edge[bend left=60]  (2+1);
			\draw[-, RoyalBlue, line width=.12em] (18+1) edge[bend left=60]  (19+1);
		\end{tikzpicture}
	\end{minipage}
	\begin{minipage}{.45\linewidth}
		\centering
		\begin{tikzpicture}
			\draw[-,BrickRed,line width = .12em] (0,0) circle (1cm);
			\node[fill=BrickRed,inner sep=0pt, minimum size=.27cm] (2) at (1,0) {};
			\node[fill=BrickRed,inner sep=0pt, minimum size=.27cm] at (-1,0) {};
			\node[fill=BrickRed,inner sep=0pt, minimum size=.27cm] (1) at (0,1) {};
			\node[fill=BrickRed,inner sep=0pt, minimum size=.27cm] at (0,-1) {};
			\foreach \a in {1,2,...,19}{
			\ifthenelse{\a=5 \OR \a=10 \OR \a=15}{}{\node[fill=RoyalBlue, inner sep = 0pt, minimum size=.12cm] (\a+1) at (360/4+\a*360/20: 1cm) {};}
			}
			
			\draw[-, RoyalBlue, line width=.12em] (1+1) edge[bend left=60]  (2+1);
			\draw[-, RoyalBlue, line width=.12em] (2+1) edge[bend left=60]  (3+1);
			\draw[-, RoyalBlue, line width=.12em] (3+1) edge[bend left=60]  (4+1);
			\draw[-, RoyalBlue, line width=.12em] (1+1) edge[bend left=80]  (4+1);
			
			\draw[-, RoyalBlue, line width=.12em] (6+1) edge[bend left=60]  (7+1);
			\draw[-, RoyalBlue, line width=.12em] (7+1) edge[bend left=60]  (8+1);
			\draw[-, RoyalBlue, line width=.12em] (8+1) edge[bend left=60]  (9+1);
			\draw[-, RoyalBlue, line width=.12em] (6+1) edge[bend left=80]  (9+1);
			
			\draw[-, RoyalBlue, line width=.12em] (11+1) edge[bend left=60]  (12+1);
			\draw[-, RoyalBlue, line width=.12em] (12+1) edge[bend left=60]  (13+1);
			\draw[-, RoyalBlue, line width=.12em] (13+1) edge[bend left=60]  (14+1);
			\draw[-, RoyalBlue, line width=.12em] (11+1) edge[bend left=80]  (14+1);
			
			\draw[-, RoyalBlue, line width=.12em] (16+1) edge[bend left=60]  (17+1);
			\draw[-, RoyalBlue, line width=.12em] (17+1) edge[bend left=60]  (18+1);
			\draw[-, RoyalBlue, line width=.12em] (18+1) edge[bend left=60]  (19+1);
			\draw[-, RoyalBlue, line width=.12em] (16+1) edge[bend left=80]  (19+1);
			
		\end{tikzpicture}
	\end{minipage}
	\caption{Lower order matchings on a cycle of length $4$. We obtain a combinatorial factor of $\theta_5(f)^q=\E[f^{(4)}(\mathcal{N}(0,1))]^{2q}$ which is compensated by a factor of $n_1^{-q-1}.$ }
	\label{fig:subleadeven}
\end{figure}
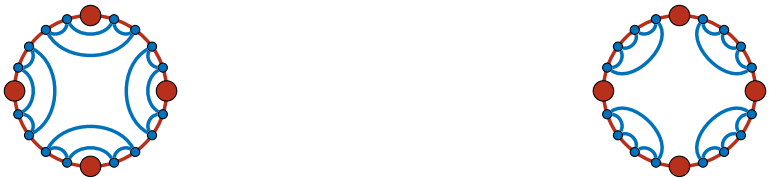

In any case, from this analysis and the one developed in \cite{BePe}*{Section 3.2}, we see that even monomials, and actually all contributions from the even part of $P_k$ do not contribute when analyzing moments of order $q\sim(\log n_1)^{1+\alpha_1}$ if $\theta_3(f)=0$.

Our next task is to show that the largest eigenvalue of $M_k$ cannot exceed $\bf{u_+}+\delta$ for any $\delta>0$ with probability arbitrarily close to 1. This is done using Proposition \ref{prop:highmoment} with the method of high moments \cite{furedi1981eigenvalues} and Markov's inequality :
\begin{equation}\label{eq:lambda1}
	\mathds{P}(\lambda_1(M_k) >{\bf{u_+}}+\delta) \leqslant \frac{\mathbb{E} \text{Tr}\, M_k^{2q}}{({\bf{u_+}}+\delta)^{2q}}. 
\end{equation}
	Using the convergence of the empirical eigenvalue distribution, one deduces that for any $\delta>0$,
	\[
	\mathds{P}\left(
	\lambda_{1}(M_k)<\bf{u_+}-\delta
	\right)\rightarrow 0.
	\]
	Now, for the other inequality one may check, from \eqref{eq:lambda1}, that we simply need to bound $\mathds{E}\mathrm{Tr}\, M^{2q}_k$ using Proposition \ref{prop:highmoment}. We can see that even for $k>\frac{\log n_1}{\log \log n_1},$ one has that $\overline{m}_{2q}^{(P_k)}\leqslant 2\mathbf{u_+}^{2q}$, by Proposition \ref{prop:highmoment}. Now, injecting this bound for the control of the largest eigenvalue we have that
	\[
	\mathds{P}\left(
	\lambda_{n_1}(M_k)>\bf{u_+}+\delta
	\right)
	\leqslant
	2n_1\left(
	\frac{\bf{u_+}}{\bf{u_+}+\delta}
	\right)^{2q}\xrightarrow[n_1\rightarrow \infty]{} 0.
	\]
	Thus the convergence of the largest eigenvalue of $M_k$ to the edge of the support implies the convergence of the largest eigenvalue of $Y^{(a_k)}$, by the bound of the spectral radius from \eqref{eq:boundspecradius}. 
\end{proof}

\paragraph{Proof of Theorem \ref{theo:sep_easy}}  Firstly, note that the contribution of odd monomials when $\theta_2(f)=0$ has been done in \cite{BePe} and is not leading even as $q\sim (\log n_1)^{1+\alpha_1}$ so that only the contribution of admissible graphs with at most long cycle is leading. In particular, the matching within the long cycle can only be performed such that it contributes $(\theta_3(f)\psi^{-1}\kappa)$ where $c$ is the length of the cycle.  We now need to finish the proof of Theorem \ref{theo:sep_easy} by allowing $k$ to grow as large as $\frac{\log n_1}{\log \log n_1}.$ The same arguments as in Proposition \ref{prop:highmoment} can be used to handle the case where $k\leqslant k_0$ and $k$ greater. This follows from the fact that any admissible graph can be interpreted as a contributing path in $m^{-q}\mathbb{E} \text{Tr}T^q $, being typical (that is with a non negligible contribution) or not. The extension from a polynomial function to an arbitrary function $f$ follows the same lines as in Theorem \ref{theo:edge}.

\subsection{Proof of Theorem \texorpdfstring{\ref{theo:edge2}}{2.5}}

We now consider a general activation function so that we can have $\theta_3(f)\neq 0$ or $\theta_2(f)\neq 0$. Firstly, note that admissible graphs with long cycles contribute to the asymptotic moment here even for finite $q$. Another possible way to maximize the number of pairwise disjoint vertices when there are both odd and even monomials is to identify the bridges in the niches associated to even monomials with the vertex in the ``long cycle'' for odd monomials as illustrated in Figure \ref{fig:matchmixed}. All other matchings lead to a lower order contribution. 

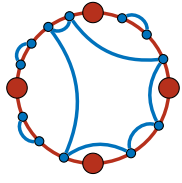
\begin{figure}[!ht]
	\centering
	\begin{tikzpicture}
		\draw[-,BrickRed,line width = .12em] (0,0) circle (1cm);
		\node[fill=BrickRed,inner sep=0pt, minimum size=.27cm] (2) at (1,0) {};
		\node[fill=BrickRed,inner sep=0pt, minimum size=.27cm] at (-1,0) {};
		\node[fill=BrickRed,inner sep=0pt, minimum size=.27cm] (1) at (0,1) {};
		\node[fill=BrickRed,inner sep=0pt, minimum size=.27cm] at (0,-1) {};
		\foreach \a in {1,2,...,4}{
			\ifthenelse{\a=5 \OR \a=10 \OR \a=15}{}{\node[fill=RoyalBlue, inner sep = 0pt, minimum size=.12cm] (\a+1) at (360/4+\a*360/20: 1cm) {};}
		}
	
		\foreach \a in {1,2,3}{
			\node[fill=RoyalBlue, inner sep = 0pt, minimum size=.12cm] (\a+2) at (180+\a*360/16: 1cm) {};
		}
		
		\foreach \a in {1,2}{
			\node[fill=RoyalBlue, inner sep=0pt, minimum size=.12cm] (\a+3) at (270+\a*360/12: 1cm) {};
		} 
	
		\foreach \a in {1,2,3}{
			\node[fill=RoyalBlue, inner sep=0pt, minimum size=.12cm] (\a+4) at (\a*360/16:1cm) {};
		}
		
		\draw[-, RoyalBlue, line width=.12em] (1+1) edge[bend left=60]  (2+1);
		\draw[-, RoyalBlue, line width=.12em] (4+1) edge[bend left=60]  (3+1);
		\draw[-, RoyalBlue, line width=.12em] (1+4) edge[bend left]  (1+1);
		\draw[-, RoyalBlue, line width=.12em] (2+1) edge[bend left]  (3+2);
		
		\draw[-, RoyalBlue, line width=.12em] (2+2) edge[bend left=60]  (1+2);
		\draw[-, RoyalBlue, line width=.12em] (3+2) edge[bend left]  (1+3);
		\draw[-, RoyalBlue, line width=.12em] (1+3) edge[bend left=60]  (2+3);
		\draw[-, RoyalBlue, line width=.12em] (2+3) edge[bend left]  (1+4);
		
		\draw[-, RoyalBlue, line width=.12em] (3+4) edge[bend left=60]  (2+4);
	\end{tikzpicture}
\caption{Example of a matching with mixed monomials. After performing perfect matchings in each niche, we identify the remaining vertices from odd niches to vertices in even niches using bridges. However, we see that $\mathds{E}[W_{11}^3]$ and $\mathds{E}[X_{11}^3]$ arises with this matching.}
\label{fig:matchmixed}
\end{figure}
Assume that a long cycle with length $2q$ has $c$ niches with an even monomial. 
Let $k_1, \ldots, k_{2q}$ be the degrees of the monomials in the niches assuming without loss of generality that $k_1, \ldots, k_{c}$ is even.  We identify the $\ell$-indices in the bridges with the vertex in each odd niche that is not matched inside the niche (equivalently the one that would be in the long blue cycle).
In this case the contribution of this cycle is then 
\begin{equation}
\frac{1}{n_1} \frac{1}{m^q n_0^{\sum_{i=1}^{2q} k_i/2}}m^q n_1^q n_0^{\sum_{i=1}^{c} (k_i/2-1)+\sum_{i=c+1}^{2q} (\frac{k_i-1}{2}) +1}=n_0^{-c/2}.
\end{equation}
Thus each such admissible graph may contribute to the moment with  a term in the order of $n_0^{-1/2}$ if a single long cycle has mixed monomials of odd and even parts of the polynomial. However, a bridge between a niche with an odd monomial and a niche between an even monomial contributes a factor of $\mathds{E}[W_{11}^3]$ or $\mathds{E}[X_{11}^3]$ which we consider to be zero. Thus, we see that the contribution from graphs with at least a long cycle is given by matchings where there are no mixed monomials within the long cycle and such a cycle (of length $2q$ ) either contributes $\theta_2(f)^{q}$ (if we only consider odd monomials) or $(n_0^{-1}\theta_3(f)\kappa)^q$ (if we only consider even monomials).  The other cases (more than one such long cycle or other matchings) are negligible.

Consider an admissible red graph with $c$ cycles of respective length $2l_1, 2l_2, \ldots, 2l_c$ and $b$ cycles of length $2$. 
The normalized contribution to $\mathbb{E} \text{Tr} M_{k}^q$ is then (from the previous sections)
\begin{equation} \label{star}
\theta_1(f)^b \prod_{i=1}^c\left ( \theta_2(f)^{2l_i}+\frac{\theta_3 (f)^{2l_i}}{ n_0} \left (\kappa_w^{2l_i}+\kappa_x^{2l_i}\right )\right )(1+o(1)).
\end{equation}

From the above one deduces that typical contributions may come from 
those admissible graphs with a single long cycle  assigned the weight $\theta_3(f) \kappa$: all the other long cycles are assigned the weight $\theta_2(f)$.  The contribution is then weighted by a factor $n_0^{-1}.$

Consider now the matrix 
\[ T:=\frac{1}{m}\left (\sqrt{\theta_1-\theta_2}\tilde Z +\sqrt{\theta_2}\frac{ \tilde W \tilde X}{\sqrt{n_0}}+2\sqrt{\theta_3 }J_2\right)\left (\sqrt{\theta_1-\theta_2}\tilde Z +\sqrt{\theta_2}\frac{ \tilde W \tilde X}{\sqrt{n_0}}+2\sqrt{\theta_3 }J_2\right)^\top .\] 
Due to the fact that the entries of $\tilde{Z}$ are centered, as that of $\tilde W$ and $\tilde X$, in order to maximize the number of pairwise distinct indices, the $\tilde Z$ edges appear in cycles of length 2 only, while that of $\tilde W\tilde X$ entries can appear in longer cycles. This is the basis for the proof of the convergence of the e.e.d.
Consider the ``odd edges'' corresponding to $\theta_3$ in the computation of the tracial moment, due to the fact that $\tilde Z$ edges appear in simple cycles and $\tilde W\tilde X$ edges in possibly longer cycles, it is not difficult to check that $J_2$ edges have to arise inside (long) cycles too. This follows from the fact that all vertices have even degree. Thus in order to maximize the number of pairwise distinct indices when some odd edge arises, the graph has to be admissible. Now if $\kappa=\max\{ \kappa_w, \kappa_x\}$, then the long cycles corresponding to $J_2$ have to be labelled by $W$ or $X$ (the one which maximizes the expectation) and the largest eigenvalue of $2\sqrt{\theta_3 }J_2J_2^\top$ is asymptotically given by $\theta_3\kappa\psi^{-1}$. 

It is not difficult to check that the contribution of admissible graphs with at least two long cycles with bridges (and other long cycle with weight $\theta_2$) does not exceed 
$$2\frac{q^2}{n_0}\Vert T\Vert^q\ll  \Vert T\Vert^q.$$ This follows from a similar analysis to that of the previous section as each long cycle would contribute a factor of $n_0^{-1}$.  
We now need to show that the contribution of non admissible graphs is negligible. First, when no $J_2$ edge arises, this follows from the same analysis as in the previous sections and as in \cite{BePe}. We need to consider apart the case where some $J_2$-edge arises and the graphs are non admissible. But then the argument follows from those similar to the last sections, when fixing the number of cycles whose weight is $\theta_3,$ since a $J_2$-edge consist exactly in computing the contribution of a cycle (in a non-admissible graphs in this case) given by $\theta_3$. 

Note that for the information-plus-noise matrix 
$$ \sqrt{\theta_1-\theta_2}\tilde Z +\sqrt{\theta_2}\frac{ \tilde W \tilde X}{\sqrt{n_0}}+2\sqrt{\theta_3 }J_2$$
 the admissible graphs with some long cycles, a single of which with weight  $\theta_3$, contribute in the large $m$ limit. The contribution of those graphs with more than one long cycle with weight $\theta_3$ can be analyzed as the non linear model ( like in \ref{star}).
This finishes the proof of Theorem \ref{theo:edge2}.
\vspace{\baselineskip}

Note that this gives a candidate to obtain a more precise result. One could try to estimate the location of the second largest eigenvalue. What combinatorics show is that \eqref{2.7} holds true: 
\begin{align*}\mathbb{E}\left (\text{Tr} M^q\right) = &(1+o(1))\sum_{I_i,I_j=0}^q\sum_{b=0}^{I_i+I_j}\sum_{L=4}^{q-b}{\mathcal{A}}(q,I_i,I_j,b; L)\theta_1(f)^b\theta_2(f)^{q-b-L}\psi^{I_i+1-q}\phi^{I_j}\theta_3(f)^L \\
&\times \left ( \kappa_w^L+ \kappa_x^L\right )\\
&+(1+o(1))\sum_{I_i,I_j=0}^q\sum_{b=0}^{I_i+I_j+1}{\mathcal{A}}(q,I_i,I_j,b)\theta_1(f)^b\theta_2(f)^{q-b}\psi^{I_i+1-q}\phi^{I_j}.
\end{align*}
Indeed, the contribution of non admissible graphs is in the order of $o(1)$ that of admissible graphs from which they are built. 
Combining two long graphs with bridges around an $i-$ and a $j-$index yields also a negligible contribution. As a consequence the spectrum has in the large $m$ limit some mass on both largest eigenvalues of $TT^\top/m$. However the bounds on fluctuations of the largest eigenvalue are not precise enough so that we can conclude about the behavior of the second one.

\bibliographystyle{abbrv}
\bibliography{bibli.bib}
\end{document}